\documentclass[11pt]{article}

\usepackage{orcidlink}
\usepackage[textwidth=480pt,textheight=650pt,centering]{geometry}

\usepackage[USenglish]{babel}
\usepackage[utf8]{inputenc}
\usepackage[T1]{fontenc}
\usepackage{lmodern}
\usepackage{dsfont}

\usepackage{amsmath,amssymb,amsfonts,amsthm, mathtools}
\usepackage{latexsym}
\usepackage{color,xcolor}
\usepackage{graphicx}
\usepackage{enumitem}
\usepackage{hyperref}
\usepackage{extarrows}
\usepackage[normalem]{ulem}
\usepackage[noadjust]{cite}

\usepackage{tikz}

\hypersetup{colorlinks=true,
 linkcolor=blue!50!black,
 citecolor=blue!50!black,
 urlcolor=blue!50!black}

\numberwithin{equation}{section}

\newtheorem{theorem}{Theorem}[section]

\newtheorem{lemma}[theorem]{Lemma}
\newtheorem{corollary}[theorem]{Corollary}
\newtheorem{proposition}[theorem]{Proposition}
\newtheorem{conjecture}[theorem]{Conjecture}
\theoremstyle{definition}
\newtheorem{remark}[theorem]{Remark}
\newtheorem{question}[theorem]{Question}

\newtheorem{example}[theorem]{Example}

\theoremstyle{definition}
\newtheorem*{remark*}{Remark}

\newcommand{\Pb}{\mathbb{P}}
\newcommand{\E}{\mathbb{E}}

\newcommand{\Z}{\mathbb{Z}}

\newcommand{\T}{\mathbb{T}}
\newcommand{\cC}{\mathcal{C}}
\newcommand{\cI}{\mathcal{I}}
\newcommand{\cL}{\mathcal{L}}

\newcommand{\supk}[1]{^{(#1)}}
\newcommand{\supN}{^{(N)}}
\newcommand{\un}{{\mathrm{un}}}
\newcommand{\bubble}{\mathsf{B}}
\newcommand{\eps}{\varepsilon}

\definecolor{darkmagenta}{rgb}{0.55, 0.0, 0.55}

\title{Percolation on Bidirected Graphs}
\author{
Yucheng Liu\,\orcidlink{0000-0002-1917-8330}\thanks{Beijing International Center for Mathematical Research,
	Peking University,
	Beijing, China 100871.
	\href{mailto:yliu135@pku.edu.cn}{yliu135@pku.edu.cn}.
	}
\and Leonardo T. Rolla\,\thanks{Department of Statistics,
	University of São Paulo,
	Brazil.
	\href{mailto:leonardo.rolla@gmail.com}{leonardo.rolla@gmail.com}.
	}
\and Yuliang Shi\,\thanks{Department of Mathematics,
	University of British Columbia,
	Vancouver, Canada V6T 1Z2.
	\href{mailto:yuliang@math.ubc.ca}{yuliang@math.ubc.ca}.
	}}

\begin{document}

\maketitle

\begin{abstract}
We study bidirected percolation on an undirected graph $G$: each edge is replaced by two oppositely directed edges, each of which is open independently with probability $p$. Two vertices are strongly connected if they are mutually reachable by open directed paths. Let $p_\mathrm{sc}$ denote the threshold for the existence of an infinite strongly connected cluster, and let $p_c$ denote the usual Bernoulli percolation threshold on $G$.

On every transitive amenable graph, we prove that $p_\mathrm{sc}=p_c$. In contrast, sufficiently spread-out modifications of every transitive non-amenable graph satisfy $p_\mathrm{sc}>p_c$. We also show that, on $\Z^d$ for all sufficiently large $d$, the expected size of the strongly connected cluster is finite at $p=p_\mathrm{sc}$.
\end{abstract}

\section{Introduction}
\label{sec:intro}

Let $G = (V, E)$ be an infinite, connected, locally finite, transitive undirected graph.
The usual Bernoulli percolation measure $\Pb_p$ on $\{0,1\}^{E}$ 
is defined by independently declaring each edge in $E$ \textbf{open} with probability $p$ and \textbf{closed} with probability $1-p$.
In this article, 
we study the Bernoulli percolation measure on the bidirected graph constructed from $G$, whose vertex set is still $V$ and whose directed edge set $\vec{E}$ is obtained by replacing each edge in $E$ with two directed edges pointing in opposite directions. 
When no confusion can arise, we also denote this new measure on $\{0,1\}^{\vec{E}}$ by $\Pb_p$.

For $x, y \in V$, we write $x \leadsto y$ if $x=y$ or if there is a directed path of open edges from $x$ to $y$. 
We say that $x$ and $y$ are \textbf{strongly connected}, denoted by $x \rightleftarrows y$, if $x \leadsto y$ and $y \leadsto x$ both hold.
We define
\begin{equation}
\cC_x^{\mathrm{out}} = \{y \in V: x \leadsto y\}, \quad
\cC_x^{\mathrm{in}} = \{y \in V: y \leadsto x\}, \quad
\cC_x^{\mathrm{sc}} = \{y \in V: x \rightleftarrows y\}
\end{equation}
to be the \textbf{outward cluster}, \textbf{inward cluster}, and \textbf{strongly connected cluster} of $x$, respectively.
For the undirected model,
we write $x \leftrightarrow y$ if there exists an open undirected path connecting $x$ and $y$, and we let $\cC^{\mathrm{un}}_x$ denote the \textbf{undirected cluster} of $x$.
A simple exploration argument (as in~\eqref{eq:exploration}) shows that
$\cC_x^{\mathrm{out}}$, $\cC_x^{\mathrm{in}}$ and $\cC_x^{\mathrm{un}}$ have the same distribution.

\begin{lemma}[{\cite[Lemma~2.1]{linusson2009notecorrelationsrandomlyoriented}}]
\label{lemma:directed}
For any $p \in [0,1]$, 
the random variable $\cC_x^{\mathrm{out}}$ (or $\cC_x^{\mathrm{in}}$) in the bidirected model
has the same distribution as the random variable $\cC_x^{\mathrm{un}}$ in the undirected model.
\end{lemma}

Fix a vertex $0 \in V$. We define the critical parameters
\begin{equation}
p_c = \inf\{p: \Pb_p(|\cC_0^{\mathrm{un}}| = \infty) > 0\}
\quad \text{ and } \quad
p_\mathrm{sc} = \inf\{p: \Pb_p(|\cC_0^{\mathrm{sc}}| = \infty) > 0\}.
\end{equation}
By Lemma~\ref{lemma:directed}, $p_c$ is also the threshold for the existence of an infinite inward or outward cluster.
The following lemma compares the strong percolation threshold $p_\mathrm{sc}$ with $p_c$.

\begin{lemma} \label{lemma:doubling}
We always have $p_c \leq p_\mathrm{sc} \leq \sqrt{p_c}$. 
\end{lemma}

\begin{proof}
The first inequality follows from the inclusion $\cC_x^{\mathrm{sc}} \subseteq \cC_x^{\mathrm{out}}$ and Lemma~\ref{lemma:directed}.
For the second inequality, we couple the bidirected model with an undirected model as follows.
We declare an undirected edge to be open exactly when both corresponding directed edges are open. This undirected percolation configuration then has distribution $\Pb_{p^2}$, and the undirected cluster of $x$ is contained in $\cC_x^{\mathrm{sc}}$. 
As a result, when $p^2 > p_c$ we must have $p > p_\mathrm{sc}$.
This proves $p_\mathrm{sc} \leq \sqrt{p_c}$.
\end{proof}

The upper bound of Lemma~\ref{lemma:doubling} is attained for trees, 
since $x \rightleftarrows y$ if and only if the unique geodesic between $x$ and $y$ is open in both directions.

\subsection{Uniqueness and structure of infinite strongly connected clusters}

Prior to this work, 
the bidirected percolation model had been studied on two-dimensional lattices.
Using a planar duality argument, 
the authors of~\cite{NoronhaMoreiraVieiraHerrmannAndradeCarmona18,Redner82,SykesEssam64} proved that $p_\mathrm{sc} = p_c$ for the two-dimensional square, triangular, and hexagonal lattices. They also conjectured that $p_\mathrm{sc} = p_c$ for finite-dimensional regular lattices.
Our first result establishes $p_\mathrm{sc} = p_c$ for all amenable transitive graphs, which include the $d$-dimensional square lattice $\Z^d$ for every $d$. 
Recall that a graph $G = (V,E)$ is called \textbf{amenable} if
\begin{equation} \label{def: amenable}
 \iota_E(G) := \inf_{\substack{K\subseteq V\\0<|K|<\infty}}
 \frac{|\partial_E K|}{|K|}=0 ,
\end{equation}
where $\partial_E K = \{\{x,y\}\in E : x\in K,\ y\notin K\}$ is the edge boundary of $K$. If $\iota_E(G) > 0$, we say that $G$ is \textbf{non-amenable}. 

\begin{theorem} \label{main_thm: equality}
 Let $G$ be an amenable transitive graph. Then 
 $p_\mathrm{sc} = p_c$.
\end{theorem}

We first present a short proof of Theorem~\ref{main_thm: equality} that uses only the uniqueness of the infinite cluster in undirected percolation. 
A more refined structure theorem for the infinite strongly connected clusters will be presented in Theorem~\ref{main_thm: structure}. 
Let $p_u$ denote the uniqueness threshold for the undirected percolation model:
\begin{equation} \label{def:uniqueness-threshold}
 p_u = \inf\{p: \Pb_p(\text{there exists a unique infinite undirected cluster}) = 1\}.
\end{equation}
For amenable graphs, it is well known that $p_u = p_c$ (see~\cite{BurtonKeane89} or~\cite[Theorem~7.6]{LyonsPeres16});
this is the only place where we use amenability.

\begin{proof}[Proof of Theorem~\ref{main_thm: equality}]
By Lemma~\ref{lemma:doubling} and~\eqref{def:uniqueness-threshold}, it suffices to prove $p_\mathrm{sc} \le p_u$. 
Let $p > p_u$ and $x,y\in V$. By the FKG inequality and Lemma~\ref{lemma:directed}, we have
\begin{equation} \label{eq:strong_lb1}
\Pb_p( x \rightleftarrows y ) 
\geq \Pb_p( x \leadsto y ) \Pb_p( y \leadsto x ) 
= \Pb_p( y \in \cC^\un_x ) \Pb_p( x \in \cC_y^\un ) 
= \Pb_p( x \leftrightarrow y )^2 .
\end{equation}
Since the infinite undirected cluster is almost surely unique when $p > p_u \ge p_c$, we get
\begin{equation} \label{eq:strong_lb2}
\Pb_p( x \leftrightarrow y )^2
\ge \Pb_p( x \leftrightarrow \infty )^2 \Pb_p( y \leftrightarrow \infty )^2
= \Pb_p( x \leftrightarrow \infty )^4
 .
\end{equation}
 We fix $x\in V$ and take a sequence $(y_n)_n$ of distinct vertices. Then the preceding uniform lower bound implies
 \begin{equation}
 \Pb_p\bigl(|\cC_x^{\mathrm{sc}}|=\infty\bigr)
 \ge \Pb_p\left(x\rightleftarrows y_n \text{ for infinitely many } n\right) 
 \ge \limsup_{n\to\infty} \Pb_p\left(x\rightleftarrows y_n\right)
 >0 ,
 \end{equation}
 so $p \ge p_\mathrm{sc}$. This concludes the proof.
\end{proof}

We now turn to the finer structure of the infinite strongly connected clusters.
We partition the vertex set $V$ into the following four categories:
\begin{equation} 
\begin{aligned}
 \cI := \{x \in V: |\cC_x^{\mathrm{in}}| = \infty, |\cC_x^{\mathrm{out}}| = \infty\}, \quad
 \cI_o := \{x \in V: |\cC_x^{\mathrm{out}}| = \infty, |\cC_x^{\mathrm{in}}| < \infty\}, \\
 \cI_i := \{x \in V: |\cC_x^{\mathrm{in}}| = \infty, |\cC_x^{\mathrm{out}}| < \infty\}, \quad
 \cI_\emptyset := \{x \in V: |\cC_x^{\mathrm{in}}| < \infty, |\cC_x^{\mathrm{out}}| < \infty\}.
\end{aligned}
\end{equation}
For a subset $A \subseteq V$, we define
\begin{equation}
 \cC_A^{\mathrm{out}} := \bigcup_{x \in A} \cC_x^{\mathrm{out}},
 \qquad
 \cC_A^{\mathrm{in}} := \bigcup_{x \in A} \cC_x^{\mathrm{in}}.
\end{equation}
For a vertex $x \in V$, we write
\begin{equation} \label{eq: cluster_of_x_truncated}
 \widetilde{\cC}^{\mathrm{un}}_x 
= 
\cC^{\mathrm{un}}_x \mathds{1}_{\{|\cC^{\mathrm{un}}_x| < \infty\}} 
 =\begin{cases}
 \cC^{\mathrm{un}}_x & (|\cC^{\mathrm{un}}_x| < \infty) \\
 \varnothing & ( |\cC^{\mathrm{un}}_x| = \infty),
 \end{cases}
\end{equation}
where the last equality defines the indicator function.
We let $\cC_\infty^{\mathrm{un}}$ denote the unique infinite cluster in the undirected percolation model whenever it exists.
The following theorem improves Theorem~\ref{main_thm: equality}
by showing the existence and uniqueness of an infinite strongly connected cluster when $p > p_u$.

\begin{theorem}
\label{main_thm: structure}
 Let $G$ be a transitive graph. Let $p > p_u$. Then, almost surely, there exists a unique infinite strongly connected cluster, and it coincides with 
 \[
 \cI = \{x \in V: |\cC_x^{\mathrm{in}}| = \infty, |\cC_x^{\mathrm{out}}| = \infty\}.
 \]
 Moreover, we have 
 \begin{equation} \label{eq:C_I}
 \cC_{\cI}^{\mathrm{out}} \overset{d}{=} \cC_{\cI}^{\mathrm{in}} \overset{d}{=} \cC_\infty^{\mathrm{un}}
 \end{equation}
 and
 \begin{equation}
 \cC_x^{\mathrm{out}} \setminus \cC_{\cI}^{\mathrm{out}} \overset{d}{=} \cC_x^{\mathrm{in}} \setminus \cC_{\cI}^{\mathrm{in}} \overset{d}{=} \widetilde{\cC}^{\mathrm{un}}_x
 \end{equation}
 for every $x\in V$.
\end{theorem}

In the proof, we do not use the strict inequality $p>p_u$, only the fact that Bernoulli bond percolation on $G$ at parameter $p$ almost surely has a unique infinite cluster (which can hold at $p=p_u$ for some graphs).

We stress that the theorem concerns the vertex set only and does not provide information on edges. 
In general, $\cC_x^{\mathrm{in}} \setminus \cC_{\cI}^{\mathrm{out}}$ does not have the same distribution as $\widetilde{\cC}^{\mathrm{un}}_x$, since the former may 
induce a disconnected subgraph of $G$.
The same applies to $\cC_x^{\mathrm{out}} \setminus \cC_{\cI}^{\mathrm{in}}$.

\begin{remark}
 Since $\cC_x^{\mathrm{out}}$ and
 $\cC_x^{\mathrm{un}}$ have the same distribution by Lemma~\ref{lemma:directed}, Theorem~\ref{main_thm: structure} yields
 \begin{equation}
 \cC_x^{\mathrm{out}} \mathds{1}_{\{|\cC_x^{\mathrm{out}}| < \infty\}}
 \overset{d}{=}
 \cC_x^{\mathrm{out}} \setminus \cC_{\cI}^{\mathrm{out}}.
 \end{equation}
 This is an equality of marginal distributions; one should not expect the two random sets to coincide almost surely. 
 Indeed, on $\{x\in \cI\}$, both sides are empty; on $\{x\in \cI_o\}$, only the left-hand side is empty; and on $\{x\in \cI_i\}$, only the right-hand side is empty. On $\{x\in \cI_\emptyset\}$, the right-hand side is contained in the left-hand side, and the inclusion may be strict.
\end{remark}

% {\color{red}
% \begin{remark*}
% Although we prove the uniqueness of infinite strongly connected clusters for $p > p_u$ via another method, it is natural to ask whether the indistinguishability and cluster-frequency arguments of Lyons and Schramm~\cite[Theorems~3.3 and~4.1]{LyonsSchramm99},
% which show that the coexistence of infinite clusters forces the infimum of the two-point function to be zero,
% can be combined with~\eqref{eq:strong_lb1}--\eqref{eq:strong_lb2} to prove uniqueness of infinite strongly connected clusters for unimodular transitive graphs.
%  The results of~\cite{LyonsSchramm99} do not apply directly, because the configuration $\eta_{\{u,v\}}=\mathds{1}_{\{u\rightleftarrows v\}}$ induced by the directed-edge configuration might not be insertion-tolerant. In general, forcing one edge of $\eta$ to be open may require changing several other edges simultaneously.
% \end{remark*}
% }

% {\color{blue}
\begin{remark}
Although we prove the uniqueness of infinite strongly connected clusters for $p > p_u$ via another method, it is natural to ask whether the indistinguishability and cluster-frequency arguments of Lyons and Schramm~\cite[Theorems~3.3 and~4.1]{LyonsSchramm99} can be combined with~\eqref{eq:strong_lb1}--\eqref{eq:strong_lb2} to prove uniqueness of infinite strongly connected clusters for unimodular transitive graphs.
In Lyons and Schramm~\cite{LyonsSchramm99}, they establish $\inf_{x, y} \Pb_p(x \leftrightarrow y) = 0$ in the nonuniqueness regime for invariant, insertion-tolerant percolation on unimodular transitive graphs.
% the coexistence of infinite clusters forces the infimum of the two-point function to be zero,
For Bernoulli percolation, Tang~\cite[Theorem~1.3]{MR4038049} extended this conclusion to all quasi-transitive graphs, confirming Schonmann's conjecture~\cite[Conjecture~3.1]{MR1833805} that the uniqueness and connectivity thresholds coincide.
However, neither result applies directly to the induced configuration $\eta_{\{u,v\}}=\mathds{1}_{\{u\rightleftarrows v\}}$: the configuration is generally dependent and need not be insertion-tolerant.
% The results of~\cite{LyonsSchramm99} do not apply directly, because the configuration $\eta_{\{u,v\}}=\mathds{1}_{\{u\rightleftarrows v\}}$ induced by the directed-edge configuration might not be insertion-tolerant. In general, forcing one edge of $\eta$ to be open may require changing several other edges simultaneously.
\end{remark}
% }

In addition to strengthening Theorem~\ref{main_thm: equality} and 
identifying the unique infinite strongly connected cluster as $\cI$,
Theorem~\ref{main_thm: structure} gives a characterization of the vertices in $\cI_o$, $\cI_i$, and $\cI_\emptyset$ in terms of their inward and outward clusters. We summarize this characterization in the following corollary.

\begin{corollary}\label{cor: classification_of_vertices}
 Under the assumptions of Theorem~\ref{main_thm: structure}, almost surely:
 \begin{enumerate}
 \item[$(i)$] $x\in \cI_o$ if and only if
 $\cC_x^{\mathrm{out}}$ intersects $\cI$ and 
 $\cC_x^{\mathrm{in}}$ is disjoint from $\cI$;
 \item[$(ii)$] $x\in \cI_i$ if and only if
 $\cC_x^{\mathrm{in}}$ intersects $\cI$ and
 $\cC_x^{\mathrm{out}}$ is disjoint from $\cI$;
 \item[$(iii)$] $x\in \cI_\emptyset$ if and only if
 $\cC_x^{\mathrm{in}}$ and
 $\cC_x^{\mathrm{out}}$ are both disjoint from $\cI$.
 \end{enumerate}
\end{corollary}

\begin{proof}
Let $p > p_u$. By Theorem~\ref{main_thm: structure}, $\cI$ is almost surely infinite.
Then, for any $x \in V$, we have
 \begin{equation}
 \{\cC_x^{\mathrm{out}}\cap \cI\neq\varnothing\}
 =
 \{x\in\cC_{\cI}^{\mathrm{in}}\}
 \subseteq
 \{|\cC_x^{\mathrm{out}}|=\infty\}.
 \end{equation}
Moreover, by~\eqref{eq:C_I} and Lemma~\ref{lemma:directed}, the left-hand side has probability
\begin{equation}
\Pb_p\bigl(x\in\cC_{\cI}^{\mathrm{in}}\bigr)
 =
 \Pb_p
 \bigl(x\in\cC_\infty^{\mathrm{un}}\bigr)
 =
 \Pb_p
 \bigl( |\cC_x ^\un | = \infty\bigr)
 =
 \Pb_p
 \bigl(|\cC_x^{\mathrm{out}}|=\infty\bigr),
\end{equation}
so $\Pb_p( \{|\cC_x^{\mathrm{out}}|=\infty\} \setminus \{\cC_x^{\mathrm{out}}\cap \cI\neq\varnothing\} ) = 0 $.
An analogous argument shows 
$\{\cC_x^{\mathrm{in}}\cap \cI\neq\varnothing\}
 \subseteq
 \{|\cC_x^{\mathrm{in}}|=\infty\}$
and
$\Pb_p( \{|\cC_x^{\mathrm{in}}|=\infty\} \setminus \{\cC_x^{\mathrm{in}}\cap \cI\neq\varnothing\} ) = 0 $. 
The desired result then follows from the definitions of $\cI_o$, $\cI_i$, $\cI_\emptyset$ and a union bound, using that $G$ is countable.
\end{proof}

By combining Theorem~\ref{main_thm: structure} with Theorem~1.3 of \cite{MR4038049}, we obtain the following corollary, which establishes that the uniqueness regime for infinite strongly connected clusters coincides with that of undirected infinite clusters on all transitive graphs.

\begin{corollary}
	Let $G$ be a transitive graph. 
	For every $p \in [0, 1]$,
	the bidirected percolation at $p$ almost surely has a unique infinite strongly connected cluster if and only if the undirected Bernoulli bond percolation at $p$ almost surely has a unique infinite cluster.
\end{corollary}

\begin{proof}
	Given Theorem~\ref{main_thm: structure}, we only need to prove the converse statement: the uniqueness of the infinite strongly connected cluster implies the uniqueness of the infinite undirected cluster.
	% fix $p \in [0,1]$ 
	Fix $p \in [0, 1]$ and suppose that bidirected percolation at $p$ almost surely has a unique infinite strongly connected cluster.
	By Lemma~\ref{lemma:directed}, uniqueness of the infinite strongly connected cluster, and the FKG inequality, for every $x,y \in V$ we have
	\begin{align}
		\Pb_p(x \leftrightarrow y) = \Pb_p(x \leadsto y) \ge \Pb_p(x \rightleftarrows y) \ge 
		\Pb_p(|\cC^{\mathrm{sc}}_x|=\infty, |\cC^{\mathrm{sc}}_y|=\infty) \ge \Pb_p(|\cC^{\mathrm{sc}}_0|=\infty)^2 > 0,
	\end{align}
	where we also used transitivity of $G$.
	By Theorem~1.3 of \cite{MR4038049}, since we have shown that $\inf_{x, y} \Pb_p(x \leftrightarrow y) > 0$, the undirected percolation almost surely has a unique infinite cluster.
\end{proof}

% \commentYL{
% A possibly cleaner way: Define
% \begin{equation} %\label{def:uniqueness-threshold}
%  p_{\mathrm{su}} = \inf\{p: \Pb_p(\text{there exists a unique infinite strongly connected cluster}) = 1\}
% \end{equation}
% and state the corollary as $ p_{\mathrm{su}}  =p_u$.

% {\color{purple}YS: I'm happy with either formulation, but I slightly prefer the current statement. Without a separate monotonicity result, the identity $p_{\mathrm{su}}=p_u$ alone might leave readers wondering whether multiple infinite SCCs can reappear for some $p>p_{\mathrm{su}}$, although our results rule this out. The current statement also covers the case $p=p_u$ whenever there is uniqueness at $p_u$, so it is stronger than equality of the thresholds alone.}
% }

\begin{remark}
Alimohammadi, Borgs, and Saberi~\cite{MR4597319} studied the same bidirected percolation model on a class of locally convergent finite expanders. They show that the threshold for a giant strongly connected cluster equals the undirected percolation threshold of the local limit, and prove a structure theorem analogous to Corollary~\ref{cor: classification_of_vertices} up to a vanishing number of vertices.

The threshold for a giant strongly connected cluster in finite graphs need not equal $p_\mathrm{sc}$ on the infinite local limit. For example, 
let $d \ge 3$ and let $\T_d$ denote the $d$-regular tree.
For $d$-regular expanders with girth tending to infinity, the former threshold is $p_c(\T_d)=(d-1)^{-1}$, whereas $p_\mathrm{sc}(\T_d)=\sqrt{p_c(\T_d)}>p_c(\T_d)$.
\end{remark}

\subsection{Size of strongly connected clusters}
\label{sec:size}

Next, we study the expected size of strongly connected clusters. 
We write
\begin{equation}
\tau_p(x,y) = \Pb_p(x \leftrightarrow y)
\end{equation}
for the two-point function of undirected percolation.
The \textbf{bubble diagram} and the \textbf{open triangle diagram} are defined to be
\begin{equation} \label{eq:def_bubble}
\bubble_p = \sum_{x\in V} \tau_p(0,x)^2 ,
\qquad
\triangle_p = \sup_{z \ne 0} \sum_{ x\ne y }
	\tau_p(0,x) \tau_p(x,y) \tau_p(y,z) ,
\end{equation}
respectively.
Note that, by taking $y=0$ and taking $z$ to be a neighbor of $0$ in the triangle, we have
\begin{equation} \label{eq:bubble_bdd}
\bubble_p = 1 + \sum_{x\ne 0} \tau_p(0,x)^2
\le 1 + \frac{1}{p} \triangle_p,
\end{equation}
so the bubble is finite when the open triangle is finite.

By the FKG inequality and Lemma~\ref{lemma:directed}, a lower bound on the size of the strongly connected cluster is given by
\begin{equation} \label{eq:susc_lb}
\E_p \vert \cC_0^{\mathrm{sc}} \vert
= \sum_{x\in V} \Pb_p (0 \rightleftarrows x)
\ge \sum_{x\in V} \Pb_p (0 \leadsto x) \Pb_p (x \leadsto 0)
= \bubble_p .
\end{equation}
The following result gives a complementary upper bound when the open triangle is small.
Our proof also establishes an upper bound for $\Pb_p(0 \rightleftarrows x)$; see Proposition~\ref{prop:sconnect_infty}.

\begin{theorem} \label{main_thm: susceptibility}
 Let $G$ be a transitive graph.
 Let $p \in (0,1]$ be such that $\triangle_p < 1$. Then
 \begin{equation} \label{eq:susc}
 \E_p \vert \cC_0^{\mathrm{sc}} \vert 
 \le \frac{\bubble_p}{p(1 - \triangle_p)}
 \le \frac{2}{p^2(1 - \triangle_p)} .
 \end{equation}
\end{theorem}

\begin{example}
For the integer lattice $\Z^d$, 
we know that $p_\mathrm{sc} = p_c$ by Theorem~\ref{main_thm: equality}. 
We also know $\triangle_{p_c} = O(1/d)$ by~\cite{HaraSlade90}.
Theorem~\ref{main_thm: susceptibility} therefore implies
\begin{equation}
\E_{p_\mathrm{sc}(\Z^d)} \vert \cC_0^{\mathrm{sc}} \vert <\infty
\end{equation}
for all $d$ sufficiently large
(we believe $d\ge 11$ suffices, since~\cite{FitznerHofstad17} shows that $\triangle_{p_c} \le 0.53562$ for $d=11$ and the upper bounds produced by the methods of~\cite{FitznerHofstad17} seem to decrease as $d$ increases).
On the other hand, it is well known that $\bubble_{p_c} = \infty$ when $d=2$ (see Appendix~\ref{app:bubble}), so~\eqref{eq:susc_lb} implies that
$\E_{p_\mathrm{sc}(\Z^2)} \vert \cC_0^{\mathrm{sc}} \vert = \infty$.
The belief that $\bubble_{p_c(\Z^d)} = \infty$ if and only if $d \le 4$ leads to the following question.
\end{example}

%\begin{example}
%For the integer lattice $\Z^d$, 
%we know that $p_\mathrm{sc} = p_c$ by Theorem~\ref{main_thm: equality}. 
%If $d \ge 11$, we also know $\triangle_{p_c} \le O(1/d) \le 0.53562$ by \cite{HaraSlade90} and \cite[(7.1)]{FitznerHofstad17}.
%%\marginnote{It is not clear that these references give 0.53562 for every $d > 11$. Need to clarify this. --Leo}
%Theorem~\ref{main_thm: susceptibility} therefore implies
% \begin{equation}
% \E_{p_\mathrm{sc}(\Z^d)} \vert \cC_0^{\mathrm{sc}} \vert <\infty
% \qquad (d \ge 11).
% \end{equation}
%On the other hand, it is well known that $\bubble_{p_c} = \infty$ when $d=2$ (see Appendix~\ref{app:bubble}), so~\eqref{eq:susc_lb} implies that
%$\E_{p_\mathrm{sc}(\Z^2)} \vert \cC_0^{\mathrm{sc}} \vert = \infty$.
%The belief that $\bubble_{p_c(\Z^d)} = \infty$ if and only if $d \le 4$ leads to the following question.
%\end{example}

\begin{question} \label{question:bubble_diagram_infinite}
 For $\Z^d$, do we have $\E_{p_\mathrm{sc}(\Z^d)} \vert \cC_0^{\mathrm{sc}} \vert = \infty$ if and only if $d \le 4$?
\end{question}

Our next result provides a sufficient condition for the strict inequality $p_c < p_\mathrm{sc}$. 
Let
\begin{align} \label{def: p_open_triangle}
p_{\triangle} 
&= \sup \Bigl\{ p : 
 \triangle_p = \sup_{z \ne 0} \sum_{ x\ne y }
	\tau_p(0,x) \tau_p(x,y) \tau_p(y,z)
 < \infty \Bigr\}
 \nonumber \\
&= \sup \Bigl\{ p : 
 \sup_{z \in V} \sum_{ x, y\in V }
	\tau_p(0,x) \tau_p(x,y) \tau_p(y,z)
 < \infty \Bigr\} 
\end{align}
denote the finite-triangle threshold.

\begin{theorem} \label{main_thm: strict}
 Let $G$ be a transitive graph. 
 If $p_c < p_\triangle$ and $\triangle_{p_c} < 1$,
 then $p_c < p_\mathrm{sc}$.
\end{theorem}

This theorem is mostly interesting for non-amenable graphs for which $p_c < p_\triangle$ can happen.
(Non-amenable graphs are conjectured to satisfy the stronger inequality $p_c < p_{2\to 2}$, where $p_{2\to 2}$ is the $L^2$ boundedness threshold~\cite[Conjecture~1.3]{Hutchcroft20}.)
We give two examples to which Theorem~\ref{main_thm: strict} applies, after its short proof.

\begin{proof}
Suppose $p_c < p_\triangle$ and $\triangle_{p_c} < 1$.
We use the continuity of the map $p \mapsto \triangle_p$ below $p_\triangle$, which is proved in Appendix~\ref{app:cont},
to pick $p_1 \in (p_c, p_\triangle)$ at which $\triangle_{p_1} < 1$ still holds.
Then, by applying Theorem~\ref{main_thm: susceptibility} at $p_1$, we learn that $\E_{p_1} \vert \cC_0^{\mathrm{sc}} \vert < \infty$.
This implies that there are almost surely no infinite strongly connected clusters at $p_1$, so we have $p_\mathrm{sc} \ge p_1 > p_c$.
\end{proof}

\begin{example} 
 Let $G_k = \T_k \times \Z$, where $\T_k$ is the $k$-regular tree. 
 For all sufficiently large $k$, we have $p_c(G_k) < p_\triangle(G_k)$ and $\triangle_{p_c(G_k)} < 1$, so Theorem~\ref{main_thm: strict} applies.
 Indeed, we have $p_c(G_k) \le 1/(k-1)$ since $\T_k$ is a subgraph of $G_k$.
On the other hand, for $p\|A\|_{2\to2}<1$, we also have $\triangle_p \le 1/(1-p\|A\|_{2\to2})^3$, where $A$ is the adjacency operator of $G_k$ and satisfies $\|A\|_{2\to2} = 2\sqrt{k-1}+2$.
 Therefore, 
 \begin{equation}
 p_\triangle (G_k)
 \ge \frac{1}{\|A\|_{2\to2}} 
 = \frac{1}{2\sqrt{k-1}+2} 
 > \frac{1}{k-1}
 \ge p_c(G_k)
 \end{equation}
 for all $k \ge 9$.
 Moreover, by Theorem~3.1 of~\cite{MR1833805}, we have $\triangle_{p_c(G_k)} < 1 $ for all large $k$.
 
\end{example}

Next, we show that, for every non-amenable graph, as long as we make it sufficiently ``spread-out,'' the resulting graph satisfies both conditions in Theorem~\ref{main_thm: strict}. It gives us a large class of graphs that have the property that $p_c < p_\mathrm{sc}$.
Given a graph $G = (V, E)$ and an integer $L \ge 1$, 
let $G_L$ be the graph with the same vertex set $V$ and edge set
\begin{equation}
 E_L := \{ \{x, y\} : x \neq y,\, 0 < d(x,y) \le L \} ,
\end{equation}
where $d(x,y)$ is the graph distance on the original $G$.

\begin{proposition} \label{cor: spread-out_non-amenable}
 Let $G$ be a non-amenable transitive graph. 
 Then there exists $L_0 \ge 1$, depending on $G$, such that
 $p_c(G_L) < p_\mathrm{sc}(G_L)$ for every $L \ge L_0$.
\end{proposition}

This leads us to the following conjecture:
\begin{conjecture} \label{question: non-amenable-transitive-pc-psc}
 If $G$ is a non-amenable transitive graph, then $p_c < p_\mathrm{sc}$.
\end{conjecture}

In light of Theorem~\ref{main_thm: equality}, 
Conjecture~\ref{question: non-amenable-transitive-pc-psc} would imply that a transitive graph $G$ is non-amenable if and only if $p_c < p_\mathrm{sc}$. Moreover, since we have established $p_\mathrm{sc} \le p_u$, Conjecture~\ref{question: non-amenable-transitive-pc-psc} 
would imply the following conjecture of Benjamini and Schramm: 

\begin{conjecture}[{\cite{BenjaminiSchramm96}}]
If $G$ is a non-amenable transitive graph, then $p_c < p_u$.
\end{conjecture}

Grimmett~\cite{MR1824275} studied a more general model in which opposite arrows may be dependent and have different opening probabilities, while the pairs associated with distinct undirected edges remain independent.
In Appendix~\ref{app:extensions}, we explain how the proofs of Theorems~\ref{main_thm: equality} and~\ref{main_thm: structure} extend to this setting under suitable assumptions on the marginals and their dependence.
We also extend the susceptibility bound of Theorem~\ref{main_thm: susceptibility} by comparison with independent arrows.

% {\color{blue}
In the physics literature, related percolation models have been studied as random resistor--diode networks, in which an edge may be absent, conduct in one direction, or conduct in both directions~\cite{Redner82,TsallisRedner83}.
Numerical studies suggest that strongly connected clusters have critical exponents different from those of standard undirected Bernoulli percolation~\cite{NoronhaMoreiraVieiraHerrmannAndradeCarmona18,NewmanGrassbergerZiff26}.
Related exponents on hierarchical lattices have been calculated using exact renormalization methods~\cite{CostaNoronhaVieiraAndradeMoreira25}.
Our Theorem~\ref{main_thm: susceptibility} makes further progress in this direction by rigorously establishing that, at $p_c = p_\mathrm{sc}$ and on $\Z^d$ with $d$ sufficiently large, the expected size of the strongly connected cluster is finite while the expected size of the ordinary percolation cluster is infinite.
Some of these models studied by physicists fall within the more general framework studied by Grimmett~\cite{MR1824275}.
% }

\subsection{Organization}
The remainder of the paper is organized as follows. 
We prove Theorem~\ref{main_thm: structure} in Section~\ref{sec:proof-of-structure}, 
prove Theorem~\ref{main_thm: susceptibility} in Section~\ref{sec:proof-of-susceptibility}, and prove Proposition~\ref{cor: spread-out_non-amenable} in Section~\ref{sec:proof-of-strict}.
Appendices~\ref{app:bubble} and~\ref{app:cont} contain the proofs of auxiliary lemmas used in Section~\ref{sec:intro}.
In Appendix~\ref{app:extensions}, we discuss generalizations of our results to the model studied in~\cite{MR1824275}.

\section{Proof of Theorem~\ref{main_thm: structure}} \label{sec:proof-of-structure}

We divide the proof into two parts. In Subsection~\ref{subsec:existence-uniqueness}, we establish the existence and uniqueness of the infinite strongly connected cluster and identify its vertex set as $\cI$. In Subsection~\ref{subsec:structure}, we prove the distributional identities relating the inward and outward clusters to undirected percolation.

\subsection{Existence and uniqueness}
\label{subsec:existence-uniqueness}

Recall that $p_u$ is the uniqueness threshold for the undirected percolation model, defined in~\eqref{def:uniqueness-threshold}.
In this subsection, we first establish the existence and uniqueness of the infinite strongly connected cluster for $p > p_u$.
The next proposition is the key step in proving the uniqueness; it states that, almost surely, every infinite outward cluster intersects every infinite inward cluster.

\begin{proposition} \label{prop-1}
 Let $x, y \in V$ be distinct. Suppose $p > p_u$.
 Then, almost surely on the event $\{|\cC^{\mathrm{out}}_x| = \infty,\; |\cC^{\mathrm{in}}_y| = \infty\}$, we have $x \leadsto y$.
\end{proposition}

The next proposition concerns the intersection of two outward clusters.

\begin{proposition} \label{prop-2}
 Let $x, y \in V$ be distinct. Suppose $p > p_u$.
 Then, almost surely on the event $\{|\cC_x^{\mathrm{out}}| = \infty,\; |\cC_y^{\mathrm{out}}| = \infty\}$, the sets $\cC_x^{\mathrm{out}} \setminus \cC_y^{\mathrm{out}}$ and $\cC_y^{\mathrm{out}} \setminus \cC_x^{\mathrm{out}}$ are both finite. The analogous statement holds for the inward clusters.
\end{proposition}

The proofs of Propositions~\ref{prop-1} and~\ref{prop-2} both use a coupling argument with undirected percolation, so we present them together.

\begin{proof}[Proof of Propositions~\ref{prop-1} and~\ref{prop-2}]
 We first prove Proposition~\ref{prop-1}. 
 Since $G$ is connected and locally finite, the set of undirected edges $E$ is countable. 
 We fix an enumeration of $E$.

 We explore $\cC_x^{\mathrm{out}}$ by revealing the states of directed edges in $\omega$ one at a time 
 while simultaneously constructing an undirected percolation configuration $\omega'$.
 Start with $A_0=\{x\}$ and declare every undirected edge \textbf{unexplored}.
 Assume that we have already completed $n$ steps of the exploration and have defined $A_{n}$.
 Then, in the $(n{+}1)$-th step, we choose the first unexplored undirected edge $e=\{u,v\}$ in the fixed enumeration such that $u \in A_n$ and $v \notin A_n$.
 Reveal $\omega(u,v)$, 
 set $e$ to be \textbf{explored}, and
 declare $e$ to be \textbf{present} in $\omega'$ if $\omega(u,v)=1$ and to be \textbf{deleted} otherwise.
 Set $A_{n+1}:=A_n\cup\{v\}$ if $\omega(u,v)=1$ and $A_{n+1}:=A_n$ otherwise.
 The exploration stops when there is no unexplored undirected edge joining $A_n$ to $A_n^c$.
 Observe that, during the exploration, each undirected edge is selected at most once.
 At each step, we reveal the state of one directed edge in $\omega$ 
 and assign the same state to the corresponding undirected edge in $\omega'$.

 We now verify that $A := \bigcup_{n\ge 0} A_n$ is the outward cluster of $x$ in $\omega$:
 It is clear that every vertex in $A$ is reachable from $x$ via a directed open path. 
 Conversely, suppose that
$
 x=v_0 \leadsto v_1 \leadsto \cdots \leadsto v_k=z
$
 is an open directed path. Once $v_i$ has been discovered, either $v_{i+1}$ is already discovered or the edge $\{v_i,v_{i+1}\}$ remains unexplored. 
 Since only finitely many edges precede it in the enumeration, we either discover $v_{i+1}$ before the edge $\{v_i,v_{i+1}\}$ is examined, or we examine this edge and discover $v_{i+1}$.
 Induction along the path therefore gives
 \begin{equation} \label{eq:exploration}
 A=\cC_x^{\mathrm{out}}.
 \end{equation}
 Moreover, the present edges used to discover vertices ensure $A$ is connected in $\omega'$. 
 Every edge joining $A$ to $A^c$ is eventually examined and declared deleted: 
 otherwise its endpoint in $A^c$ would be discovered.
 Hence $A$ is an undirected cluster of $\omega'$.

 Now, we assume that $y\notin A$.
 We do a second exploration to find all the vertices that can reach $y$ inside $A^c = V \setminus A$.
 Throughout this second exploration, every undirected edge selected during the construction of $A$ is regarded as already explored.
 Let $B_0=\{y\}$ and iteratively construct $B_1,B_2,\ldots$, where $B_n$ is the collection of discovered vertices after the $n$-th step.
 Given $B_n$, we choose the first unexplored undirected edge $e=\{u,v\}$ such that $u \in B_n$ and $v \in A^c \setminus B_n$.
 Reveal $\omega(v,u)$, 
 set $e$ to be explored, and
 declare $e$ to be present in $\omega'$ if $\omega(v,u)=1$ and to be deleted otherwise.
 Set $B_{n+1}:=B_n\cup\{v\}$ if $\omega(v,u)=1$ and $B_{n+1}:=B_n$ otherwise. 
Then we have
 \begin{equation}
 B := \bigcup_{n\ge 0} B_n = \{z\in A^c:z \stackrel{A^c}{\leadsto} y\},
 \end{equation}
 where $z \stackrel{A^c}{\leadsto} y$ means that there exists a directed open path from $z$ to $y$ that is contained in $A^c$.
 Moreover, the present edges revealed during this exploration ensure that $B$ is connected in $\omega'$.

 Consider the event
 \begin{equation}
 E
 :=
 \bigl\{
 |\cC_x^{\mathrm{out}}|=\infty,\,
 |\cC_y^{\mathrm{in}}|=\infty,\,
 x\nrightarrow y
 \bigr\}.
 \end{equation}
 On $E$, every open directed path reaching $y$ avoids $A$,
 so $B=\cC_y^{\mathrm{in}}$.
 Thus, on $E$, the set $A$ is an infinite undirected cluster of
 $\omega'$, while the infinite connected set $B\subseteq A^c$ lies in a different infinite cluster.
 By the uniqueness of the infinite cluster in undirected percolation when $p>p_u$, 
 we can conclude $\Pb(E)=0$, proving Proposition~\ref{prop-1}.

 To prove Proposition~\ref{prop-2}, we first explore
 $A=\cC_x^{\mathrm{out}}$ as above, and then we explore the outward cluster of
 $y$ inside $A^c$ to construct $B$. In this case, it is straightforward to verify that
 \begin{equation}
 B = \cC_y^{\mathrm{out}}\setminus A.
 \end{equation}
 Now, it follows from the uniqueness of the infinite cluster in the undirected percolation model that at most one of $A=\cC_x^{\mathrm{out}}$ and $B$ is infinite almost surely.
 Interchanging the roles of $x$ and $y$ shows that at most one of $\cC_y^{\mathrm{out}}$ and $\cC_x^{\mathrm{out}} \setminus \cC_y^{\mathrm{out}}$ is infinite almost surely.
 Therefore, on the event that both $\cC_x^{\mathrm{out}}$ and $\cC_y^{\mathrm{out}}$ are infinite, almost surely both $\cC_x^{\mathrm{out}} \setminus \cC_y^{\mathrm{out}}$ and $\cC_y^{\mathrm{out}} \setminus \cC_x^{\mathrm{out}}$ are finite, proving Proposition~\ref{prop-2}.
 Finally, reversing all directed edges in the model
 gives the corresponding statement for inward clusters.
\end{proof}

The next lemma establishes the first part of Theorem~\ref{main_thm: structure}.

\begin{lemma} \label{lemma:I}
Let $p > p_u$. Then, almost surely, there exists a unique infinite strongly connected cluster, and it coincides with 
 \[
 \cI = \{x \in V: |\cC_x^{\mathrm{in}}| = \infty, |\cC_x^{\mathrm{out}}| = \infty\}.
 \]
\end{lemma}

\begin{proof}
Let $p >p_u$.
We first prove that almost surely $\cI$ is the unique infinite strongly connected cluster.
We claim that for every $p \in (p_u,1]$,
\begin{equation}
 \Pb_p
 \left( \left\{ 
 \begin{array}{c}
 \text{either $\cI$ is empty or $\cI$ is the unique} \\
 \text{infinite strongly connected cluster}
 \end{array}
 \right\} \right) = 1.
\end{equation}
Indeed, observe that $\cI^c$ has empty intersection with every infinite strongly connected cluster by definition. Next, if $\cI$ is non-empty, i.e., there exists $x$ such that $|\cC_x^{\mathrm{in}}| = |\cC_x^{\mathrm{out}}| = \infty$,
then it follows from the local finiteness of $G$ that there exist infinitely many vertices $y \in \cC_x^{\mathrm{out}} \cup \cC_x^{\mathrm{in}}$ such that $|\cC_y^{\mathrm{in}}| = |\cC_y^{\mathrm{out}}| = \infty$.
(Any $y$ on an infinite directed path starting from $x$ has this property.)
Thus, $\cI$ is infinite when it is non-empty. 
Moreover, Proposition~\ref{prop-1} implies that any two vertices in $\cI$ can be connected in both directions almost surely, so $\cI$ is a strongly connected cluster.
Combining these observations proves the claim.

By the claim, it suffices to show that almost surely $\cI$ is non-empty. Note that the event $\cI \ne \varnothing$ is invariant under the automorphism group of $G$. Therefore, by ergodicity, it suffices to show $\Pb_p(\cI \neq \varnothing) > 0$. This is true since when $p > p_u \ge p_c$, we have 
\begin{equation}
\label{ineq: FKG_structure}
 \Pb_p\left(|\cC_x^{\mathrm{out}}| = \infty, |\cC_x^{\mathrm{in}}| = \infty\right) \ge \Pb_p\left(|\cC^{\mathrm{un}}_x| = \infty\right)^2 > 0,
\end{equation}
by the FKG inequality and Lemma~\ref{lemma:directed}.
\end{proof}

\subsection{Structure}
\label{subsec:structure}

We now prove the second part of Theorem~\ref{main_thm: structure}.
The idea is to approximate the outward cluster of $\cI$, inside any fixed finite observation set, by the outward cluster of a large deterministic ball placed far away. With high probability, the ball intersects $\cI$; the following lemma shows that the extra vertices reached from the ball are confined to a larger surrounding ball. Placing that larger ball outside the observation set makes the two outward clusters agree there with high probability. The advantage is that exploration from a deterministic ball allows us to compare the directed and undirected models.

We will need the following lemma.
For $A \subseteq V$, we write
\begin{equation}
 \cC_A^{\mathrm{out}}
 := \bigcup_{y\in A} \cC_y^{\mathrm{out}}.
\end{equation}
We define $\cC_A^{\mathrm{in}}$ and $\cC_A$ similarly.
Let $B(x,r)$ denote the set of vertices at graph distance at most $r$ from $x$.

\begin{lemma}\label{lemma: large_box_intersection}
Let $p > p_u$.
Then there exists a sequence of positive integers $(K(m))_{m\ge 1}$ with $K(m) \ge m$ for which
\begin{equation*}
\lim_{m\to \infty} \sup_{x\in V}
 \Pb_p
 \left(
 \cC_{B(x,m)}^{\mathrm{out}}\setminus B(x,K(m))
 \neq
 \cC_{\cI}^{\mathrm{out}}\setminus B(x,K(m))
 \right)
 = 0 .
 \end{equation*}
\end{lemma}

\begin{remark}
By using recent results on the supercritical sharpness of percolation on vertex-transitive graphs~\cite[Theorem~1]{DiskinEasoRadhakrishnanSudakovTassion26}, we can prove a quantitative version of Lemma~\ref{lemma: large_box_intersection} with $K(m) = K_{G,p} m$ where $K_{G,p} \ge 2$.
This stronger version is not needed in the proof of our main results.
\end{remark}

\begin{proof}[Proof of Lemma~\ref{lemma: large_box_intersection}]
Observe that, by transitivity, the invariance of $\Pb_p$ under any graph automorphism $\phi$,
and the identity $\phi(\cI(\omega))=\cI(\phi(\omega))$ for any configuration $\omega$,
the probability in the statement of the lemma is the same for every $x \in V$. Therefore, we can ignore the supremum and treat $x$ as an arbitrary fixed site.

Fix $p>p_u$ and let $\epsilon>0$.
We will show that there exists a sequence $K(m)$ independent of $\epsilon$ and a number $m_0$ for which
 \begin{equation} \label{eq:limitlessepsilon}
 \Pb_p
 \left(
 \cC_{B(x,m)}^{\mathrm{out}}\setminus B(x,K(m))
 \neq
 \cC_{\cI}^{\mathrm{out}}\setminus B(x,K(m))
 \right)
 \le \epsilon
 \qquad (m\ge m_0).
 \end{equation}
 Throughout the proof, we work on the event 
 \begin{align}
 \mathcal{E} := \, &\bigl\{ 
 \cI \text{ is the unique infinite strongly connected cluster}
 \bigr\} \, \bigcap \nonumber \\
 &\hspace{2em}\bigcap_{v \in V} \Bigl\{ |\cC_v^{\mathrm{out}}| = \infty \implies |\cC_y^{\mathrm{out}}\setminus \cC_v^{\mathrm{out}}| < \infty \text{ for all } y \in V \Bigr\}.
 \end{align}
By Lemma~\ref{lemma:I} and Proposition~\ref{prop-2}, this event has probability $1$.
Since $\cI$ is nonempty and since $G$ is connected, 
there exists $m_0$ such that
 \begin{equation}
 \Pb_p\left(\cI \cap B(x,m) = \varnothing\right) 
 \le \epsilon/2 
 \end{equation}
 for all $m\ge m_0$.
On the event $\{\cI \cap B(x,m) \neq \varnothing\}$, we take $v \in \cI \cap B(x,m)$, so that
 $  \cC_{\cI}^{\mathrm{out}} = \cC_v^{\mathrm{out}} \subseteq \cC_{B(x,m)}^{\mathrm{out}} $
and
 \begin{equation}
 \cC^{\mathrm{out}}_{B(x,m)} \setminus \cC^{\mathrm{out}}_{\cI} = \bigcup_{y \in B(x, m)} \left( \cC_y^{\mathrm{out}} \setminus \cC_v^{\mathrm{out}} \right).
 \end{equation}
Under $\mathcal{E}$, since $\cC^{\mathrm{out}}_v$ is infinite, the set $\cC_y^{\mathrm{out}} \setminus \cC^{\mathrm{out}}_v$ is necessarily finite for every $y$. 
Taking the finite union then shows that $\cC^{\mathrm{out}}_{B(x, m)} \setminus \cC_{\cI}^{\mathrm{out}}$ is finite.
Therefore,
 \begin{equation}
 \Pb_p \left( 
 \cI \cap B(x,m) \neq \varnothing, \, \cC^{\mathrm{out}}_{B(x, m)} \setminus B(x, k) \neq \cC^{\mathrm{out}}_{\cI} \setminus B(x, k) \right) \rightarrow 0 
 \text{ \ as \ } k \to \infty.
 \end{equation}
Note that the above probability does not depend on $\epsilon$ or on the choice of a site $v$; it depends only on $m$ and $k$.
For each $m$, we choose $K=K(m)$ large enough that the probability above is at most $1/m$ for every $k \geq K$.

Finally, we enlarge the previous $m_0$ so that $1/m_0 < \epsilon/2$. A union bound over the two cases then gives the desired \eqref{eq:limitlessepsilon}.
\end{proof}

\begin{proof}
[Proof of Theorem~\ref{main_thm: structure}]
Our goal is to prove
\begin{align} \label{eq:sec2_goal}
 \text{ (i) } \cC_{\cI}^{\mathrm{out}} 
 \overset{d}{=} \cC_\infty^{\mathrm{un}} ; 
 \quad 
 \text{ (ii) } \cC_x^{\mathrm{out}} \setminus \cC_{\cI}^{\mathrm{out}} 
 \overset{d}{=} 
 \widetilde{\cC}_x ^{\mathrm{un}}
 \text{ for every } x \in V.
\end{align}
The analogous distributional equalities for $\cC_{\cI}^{\mathrm{in}}$ and $\cC_x^{\mathrm{in}} \setminus \cC_{\cI}^{\mathrm{in}}$ follow by reversing the orientation of all directed edges.

We first prove item (i).
It suffices to show that, 
for every finite $W \subseteq V$ and every $\epsilon > 0$,
we have
\begin{equation} \label{eq:dist_goal1}
 d_{\mathrm{TV}}
 \left(
 \cL(\cC_{\cI}^{\mathrm{out}} \cap W),
 \cL(\cC_\infty^{\mathrm{un}} \cap W)
 \right)
 \le 2\epsilon ,
\end{equation}
where $d_{\mathrm{TV}}$ denotes the total variation distance between two probability measures.
Let $(K(m))_{m\ge1}$ be given by Lemma~\ref{lemma: large_box_intersection}.
For every $m$, choose $x_m \in V$ such that $d(x_m,W)\ge K(m){+}m.$
In particular, 
we have $W \subseteq B(x_m, K(m))^c$.
Using Lemma~\ref{lemma: large_box_intersection}, by choosing $m$ large enough depending on $\epsilon$, we have
\begin{align}
&\Pb_p
 \left( 
 \cC_{B(x_m,m)}^{\mathrm{out}} \cap W = \cC_{\cI}^{\mathrm{out}} \cap W
 \right)
 \nonumber \\
&\ge
 \Pb_p
 \left(
 \cC_{B(x_m,m)}^{\mathrm{out}} \setminus B(x_m,K(m))
 =
 \cC_{\cI}^{\mathrm{out}} \setminus B(x_m,K(m))
 \right)
 \ge 1-\epsilon.
\end{align}
Consequently,
\begin{equation} \label{eq:directed-TV}
 d_{\mathrm{TV}}
 \left(
 \cL(\cC_{B(x_m,m)}^{\mathrm{out}} \cap W),
 \cL(\cC_{\cI}^{\mathrm{out}} \cap W)
 \right)
 \le \epsilon .
\end{equation}
On the other hand, we claim that, by possibly further increasing $m$, we can have
\begin{equation}
\label{eq:undirected-convergence}
 d_{\mathrm{TV}}
 \left(
 \cL(\cC^{\mathrm{un}}_{B(x_m,m)} \cap W),
 \cL(\cC^{\mathrm{un}}_\infty \cap W)
 \right) \le \epsilon
\end{equation}
in the undirected model.
Indeed, 
if $\cC^{\mathrm{un}}_{B(x_m,m)} \cap W$ is different from $\cC^{\mathrm{un}}_\infty \cap W$
and if $\cC^{\mathrm{un}}_\infty$ intersects $B(x_m,m)$,
then there must be a finite cluster that connects $B(x_m, m)$ to $W$.
Thus, we have
\begin{multline}
\label{eq:probclusterdiff}
 \Pb_p\left(
 \cC^{\mathrm{un}}_{B(x_m,m)} \cap W \neq \cC^{\mathrm{un}}_\infty \cap W
 \right) \le \;
 \Pb_p\left(
 \cC^{\mathrm{un}}_\infty \cap B(x_m,m) = \varnothing
 \right) \\
 \hspace{20mm} + \sum_{w\in W}
 \Pb_p\bigl(
 |\cC^{\mathrm{un}}_w|<\infty, \, \cC^{\mathrm{un}}_w \cap B(x_m,m) \neq \varnothing
 \bigr).
\end{multline}
Since $p>p_u\ge p_c$, the first term on the right-hand side goes to zero as $m\to\infty$.
Moreover, since $W$ is finite and since $d\bigl(W,B(x_m,m)\bigr) \ge K(m) \ge m \to \infty$, the second term goes to zero as well,
proving~\eqref{eq:undirected-convergence}.

Finally, note that $\cC_{B(x_m,m)}^{\mathrm{out}} \overset{d}{=} \cC^{\mathrm{un}}_{B(x_m,m)}$, which can be shown using the exploration used to construct $A$ in~\eqref{eq:exploration}, with the initial set replaced by $A_0=B(x_m,m)$.
Combining this with~\eqref{eq:directed-TV} and~\eqref{eq:undirected-convergence},
we obtain the desired~\eqref{eq:dist_goal1}.

To prove identity (ii) in~\eqref{eq:sec2_goal},
we first show that for every finite set $D\subseteq V$ and every $x\in V$,
\begin{equation} \label{eq:exploration-identity}
 \cC_x^{\mathrm{out}}\setminus \cC_D^{\mathrm{out}}
 \overset{d}{=} \widetilde{\cC}^{\mathrm{un}}_{x, D} := \cC^{\mathrm{un}}_{x} \mathds{1}_{\{x \not \leftrightarrow D\}}.
\end{equation}
We explore the vertices of $\cC_D^{\mathrm{out}}$ 
in the same way as we explored $\cC_x^{\mathrm{out}}$ in the proof of Proposition~\ref{prop-1}.
Fix an enumeration of the undirected edges of $G$.
Initially, set $A_0=D$ and declare every undirected edge unexplored.
Assume that, for some $n\ge 0$, we have constructed $A_n$ and have declared exactly $n$ undirected edges explored.
In the $(n{+}1)$-th step, we choose the first unexplored undirected edge $e=\{u,v\}$ in the fixed enumeration such that $u\in A_n$ and $v\notin A_n$ and mark $e$ as explored.
If $\omega(u,v)=1$, let $A_{n+1}:=A_n\cup\{v\}$; otherwise, let $A_{n+1}:=A_n$.
The process terminates when there are no such unexplored edges.
Let $A:=\cup_{n\ge0}A_n$ be the set of discovered vertices, where, if the procedure terminates at time $N$, we set $A_n:=A_N$ for all $n>N$.

Then, we explore the vertices of $\cC_x^{\mathrm{out}} \setminus \cC_D^{\mathrm{out}}$.
If $x \in A$, set $B := \varnothing$ and terminate the exploration.
If $x \notin A$, we follow the same procedure as above, except that we set $B_0:=\{x\}$ and regard
every undirected edge explored during the previous exploration as already explored.
This process gives $B_0, B_1, \ldots$. Let $B:=\cup_{n\ge0} B_n$ be the set of discovered vertices.

We make the following two observations: 
(a) each undirected edge is explored at most once and 
(b) after the first exploration, every undirected edge $e = \{u, v\}$ such that $u \in A$ and $v \notin A$ is explored with $(u, v)$ being closed. 
Combining these two observations, it is straightforward to verify that 
$\cC_x^{\mathrm{out}} \setminus \cC_D^{\mathrm{out}} = B$ and 
$\widetilde{\cC}^{\mathrm{un}}_{x, D}$ has the same law as $B$, which proves~\eqref{eq:exploration-identity}.

We now continue to prove item (ii). It suffices to show that, for every finite $W\subseteq V$ and every $\epsilon> 0$, we have
\begin{equation} \label{eq:dist_goal2}
 d_{\mathrm{TV}}\left(
 \cL\left(
 \bigl(\cC_x^{\mathrm{out}}\setminus \cC_{\cI}^{\mathrm{out}}\bigr)
 \cap W
 \right),
 \cL(\widetilde{\cC}_x^{\mathrm{un}}\cap W)
 \right)
 \le 2\epsilon.
\end{equation}
Let $m$ be a positive integer and $K(m)$ be given by Lemma~\ref{lemma: large_box_intersection}.
For each $m$, choose $z_m \in V$ such that $d(z_m,W\cup\{x\}) > K(m)+m$. 
Then we have
$W \cup \{x\} \subseteq B(z_m, K(m))^c$.
By Lemma~\ref{lemma: large_box_intersection}, for all sufficiently large $m$,
\begin{align}
& \Pb_p\left( \cC^{\mathrm{out}}_{B(z_m,m)} \cap W 
	= \cC_{\cI}^{\mathrm{out}} \cap W \right) 
\nonumber \\
&\ge \Pb_p\left( \cC^{\mathrm{out}}_{B(z_m,m)} \setminus B(z_m,K(m)) = \cC_{\cI}^{\mathrm{out}} \setminus B(z_m,K(m)) \right) 
 \ge 1-\epsilon.
\end{align}
By subtracting the same set from $\cC_x ^{\mathrm{out}} \cap W$, we get
\begin{equation} \label{eq:first-TV-bound}
 \Pb_p\left( 
 \bigl(\cC_x^{\mathrm{out}} \setminus \cC_{B(z_m,m)}^{\mathrm{out}}\bigr) \cap W 
 \neq \bigl(\cC_x^{\mathrm{out}} \setminus \cC_{\cI}^{\mathrm{out}}\bigr) \cap W 
 \right) \le \epsilon.
\end{equation}
By~\eqref{eq:exploration-identity} with $D = B(z_m,m)$, we also have
\begin{equation} \label{eq:TV-identity}
 \bigl(\cC_x^{\mathrm{out}} \setminus \cC_{B(z_m,m)}^{\mathrm{out}}\bigr) \cap W \overset{d}{=} \widetilde{\cC}^{\mathrm{un}}_{x, B(z_m,m)} \cap W
\end{equation}
for all $W$ and $m$.

On the other hand, we claim that, by possibly further increasing $m$, we have
\begin{equation} \label{eq:second-TV-bound}
 \Pb_p\left( \widetilde{\cC}^{\mathrm{un}}_{x, B(z_m,m)} \cap W \neq \widetilde{\cC}^{\mathrm{un}}_x \cap W \right) \le \epsilon
\end{equation}
in the undirected model.
Recall the definition of $\widetilde{\cC}^{\mathrm{un}}_{x, B(z_m,m)}$ from~\eqref{eq:exploration-identity} and the definition of $\widetilde{\cC}^{\mathrm{un}}_x$ from~\eqref{eq: cluster_of_x_truncated}.
We can split the event in~\eqref{eq:second-TV-bound} into two cases: (i) if $x \leftrightarrow \infty$, then $\widetilde{\cC}^{\mathrm{un}}_x = \varnothing$ by definition, and thus we must have $x \not\leftrightarrow B(z_m, m)$; (ii) if $x \leftrightarrow B(z_m, m)$, then $\widetilde{\cC}^{\mathrm{un}}_{x, B(z_m, m)} = \varnothing$ by definition, and thus we must have $x \not\leftrightarrow \infty$.
Therefore, we have
\begin{align}
 \Pb_p\left( \widetilde{\cC}^{\mathrm{un}}_{x, B(z_m,m)} \cap W \neq \widetilde{\cC}^{\mathrm{un}}_x \cap W \right) 
 &\le 
 \Pb_p \left( 
 \begin{array}{c}
 x \leftrightarrow \infty, \\
 x \not\leftrightarrow B(z_m,m)
 \end{array}
 \right) + 
 \Pb_p \left( 
 \begin{array}{c}
 x \not\leftrightarrow \infty, \\
 x \leftrightarrow B(z_m,m) 
 \end{array} 
 \right) \nonumber \\[0.2em]
 &\le \Pb_p \Bigl( B(z_m,m) \not\leftrightarrow \infty \Bigr) \, + \, \Pb_p \Bigl( m \le \mathrm{diam}(\cC_x^{\mathrm{un}}) < \infty \Bigr),
\end{align}
where we used $p > p_u$ for the first term and
$d(x,z_m) > K(m) + m \ge 2m$ for the second term in the last line.
This upper bound goes to $0$ as $m\to \infty$, since $p > p_u \ge p_c$.

Finally, we combine~\eqref{eq:first-TV-bound},
\eqref{eq:TV-identity}, and
\eqref{eq:second-TV-bound} with the triangle inequality to get
the desired~\eqref{eq:dist_goal2}.
This completes the proof of item (ii) and the proof of Theorem~\ref{main_thm: structure}.
\end{proof}

\section{Proof of Theorem~\ref{main_thm: susceptibility}} \label{sec:proof-of-susceptibility}

In this section, we prove Theorem~\ref{main_thm: susceptibility} using a diagrammatic method. 
Our main tool is the BK inequality.
Let $(A_i)_{i = 1}^{n}$ be a finite sequence of events in $\{0,1\}^{\vec{E}}$.
The event of \textbf{disjoint occurrence} of $(A_i)_{i = 1}^{n}$ is
\begin{equation}
 A_1 \circ \cdots \circ A_n := \left\{ \omega \in \{0,1\}^{\vec{E}} : 
 \begin{array}{c}
 \exists \text{ disjoint } \vec{E}_1, \ldots, \vec{E}_n \subseteq \vec{E} \\
 \text{ such that } [\omega]_{\vec{E}_i} \subseteq A_i \text{ for each } i
 \end{array} 
 \right\},
\end{equation}
where $[\omega]_{\vec{E}_i}$ is the event that consists of all $\omega' \in \{0,1\}^{\vec{E}}$ such that $\omega'$ agrees with $\omega$ on the directed edges in $\vec{E}_i$. 
The BK inequality says that
$\Pb_p(A_1 \circ \cdots \circ A_n) \le \Pb_p(A_1) \cdots \Pb_p(A_n)$.

The following lemma will be useful in the computation of $\E_p \vert \cC_0^{\mathrm{sc}} \vert $.

\begin{lemma} \label{lem:3pt_decomp}
For any vertices $x,y$, we have
\begin{equation}
\begin{aligned} 
 \{ 0 \leadsto x \leadsto y \} \,
 \subseteq \, &\bigl\{ ( 0 \leadsto x ) \circ ( x \leadsto y ) \bigr\} \, \bigcup \\
	& \; \; \bigcup_{ u, v } 
	\bigl\{ ( 0\leadsto u ) \circ ( u \leadsto v ) \circ ( v \leadsto y )
		\circ ( v \leadsto x \leadsto u ) \bigr\},
\end{aligned}
\end{equation}
where the union is over all pairs of vertices $u,v$ satisfying the constraints
$u \ne v$, $u \ne y$, and $v \ne 0$.
\end{lemma}

\begin{proof}
We bound the event $\{ 0 \leadsto x \leadsto y \}
\setminus \{ ( 0\leadsto x ) \circ ( x \leadsto y ) \}$.
On this event, there exists a pair of non-trivial self-avoiding open (directed) paths $\gamma_1 : 0\leadsto x$ and $\gamma_2 : x\leadsto y$,
and they must share at least one (directed) edge.
Let $(u,u')$ denote the first such edge along $\gamma_1$.
We write $\gamma_2[u,y]$ for the part of $\gamma_2$ from $u$ to $y$, 
and we find some vertex $v$ on $\gamma_2[u,y]$ such that $\gamma_1[v,x]$ is (edge-)disjoint from $\gamma_2[u,y]$; this can be established by taking the last point along $\gamma_1$ that also belongs to $\gamma_2[u,y]$.
Then, since $\gamma_1[0,u]$ is disjoint from $\gamma_2$ by the choice of the first shared edge, the event
\begin{equation} \label{eq:3pt_decomp_pf}
( 0\leadsto u ) \circ ( u \leadsto v ) \circ ( v \leadsto y )	\circ ( v \leadsto x \leadsto u ) 
\end{equation}
holds; see Figure~\ref{fig:3pt_decomp}. The desired result then follows from a union bound; by our choice, we have $u \ne v$, $u\ne y$, and $v \ne 0$.
\end{proof}

\begin{figure}[h]
\center{
\begin{tikzpicture}[>=stealth, line width=0.75pt]
\node (x) at (0,1.6) {$0$};
\node (u) at (2,1.6) {$u$};
\node (v) at (2,0) {$v$};
\node (y) at (0,0) {$y$};
\node (o) at (5,0.8) {$x$};

\draw[->] (x) -- (u);
\draw[->] (u) -- (v);
\draw[->] (v) -- (y);
\draw[->]
 (o)
 .. controls (4.5,0.4) and (4,0.2)
 .. (3.5, 0.8)
 .. controls (3.2,1.2) and (2.5,1.6)
 .. (u);
\draw[->]
 (v)
 .. controls (2.5,0) and (3.2, 0.4)
 .. (3.5, 0.8)
 .. controls (4,1.4) and (4.5,1.2)
 .. (o);
\end{tikzpicture}

\caption{
Schematic picture for the occurrence of the event~\eqref{eq:3pt_decomp_pf}. 
The paths $0 \leadsto u$ and $v \leadsto x$ come from $\gamma_1$.
The path $x \leadsto u \leadsto v \leadsto y$ comes from $\gamma_2$.
}
\label{fig:3pt_decomp}
}
\end{figure}
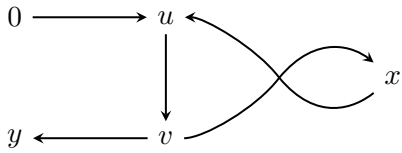

Recall 
$\tau_p(x,y) = \Pb_p(x \leftrightarrow y)
= \Pb_p( x \leadsto y)$
denotes the two-point function.
Using Lemma~\ref{lem:3pt_decomp} with $y=0$ and the BK inequality, we have
\begin{equation} \label{eq:sconnect_base}
\Pb_p( 0 \rightleftarrows x)
\le \tau_p(0,x)^2 + \sum_{ \substack{ u,v } }
	\tau_p(0,u) \tau_p (u,v) \tau_p(v,0) \Pb_p(v \leadsto x \leadsto u),
\end{equation}
where the sum is over all pairs of vertices $u,v$ satisfying 
$0 \ne u \ne v \ne 0$.
We can apply Lemma~\ref{lem:3pt_decomp} to the last factor in the sum again. 
For our final bound, we define the functions
\begin{equation}
\begin{aligned}
 U_p\supk{1}(z; u,v) &:= \tau_p(z,u) \tau_p (u,v) \tau_p(v,z) ,\\
 U_p\supN(z; u',v') &:= \sum_{ \substack{ u, v } }
 U_p\supk{N-1}(z; u, v) \tau_p(v, u') \tau_p (u',v') \tau_p(v',u)
 \qquad (N\ge 2),
\end{aligned}
\end{equation}
where the sum is over all pairs of vertices $u,v$ satisfying $u \ne v$, $u' \ne u$, and $v' \ne v$.
Also, define
\begin{equation}
R_p\supN(x) := \sum_{ u\ne v }
	U_p\supk{N}(0; u,v) \Pb_p(v \leadsto x \leadsto u)
	\qquad(N\ge 1),
\end{equation}
which gives the remainder after $N$ applications of Lemma~\ref{lem:3pt_decomp}.

\begin{lemma} \label{lem:sconnect_N}
Let $x$ be a vertex.
For any $N \ge 0$, we have
\begin{equation} \label{eq:sconnect_N}
\Pb_p(0 \rightleftarrows x)
\le \tau_p(0,x)^2 
	+ \biggl( \sum_{n=1}^N 
		\sum_{ u\ne v }
		U_p\supk{n}(0; u,v) \tau_p(u,x) \tau_p(v,x) \biggr)
	+ R_p\supk{N+1}(x) .
\end{equation}
\end{lemma}

The first three terms of~\eqref{eq:sconnect_N} are depicted in Figure~\ref{fig:sconnect_expansion}.

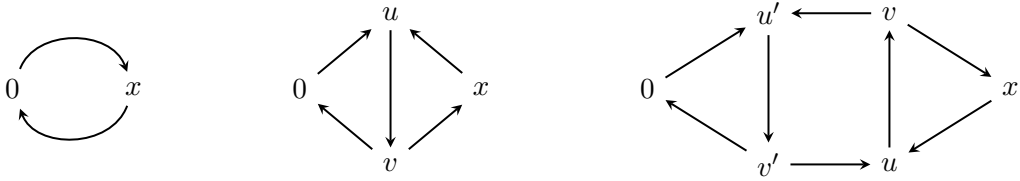
\begin{figure}[h]
\centering
\begin{tikzpicture}[>=stealth, line width=0.75pt]

\node (bo) at (0,0) {$0$};
\node (bx) at (1.6,0) {$x$};

\draw[->] (bo)
 .. controls (0.3,0.8) and (1.3,0.8) .. (bx);
\draw[->] (bx)
 .. controls (1.3,-0.8) and (0.3,-0.8) .. (bo);

\begin{scope}[shift={(3.8,0)}]
 \node (co) at (0,0) {$0$};
 \node (cu) at (1.2,1) {$u$};
 \node (cv) at (1.2,-1) {$v$};
 \node (cx) at (2.4,0) {$x$};

 \draw[->] (co) -- (cu);
 \draw[->] (cu) -- (cv);
 \draw[->] (cv) -- (co);
 \draw[->] (cv) -- (cx);
 \draw[->] (cx) -- (cu);
\end{scope}

\begin{scope}[shift={(8.4,0)}]
 \node (o) at (0,0) {$0$};
 \node (up) at (1.6,1) {$u'$};
 \node (vp) at (1.6,-1) {$v'$};
 \node (v) at (3.2,1) {$v$};
 \node (u) at (3.2,-1) {$u$};
 \node (x) at (4.8,0) {$x$};

 \draw[->] (o) -- (up);
 \draw[->] (up) -- (vp);
 \draw[->] (vp) -- (o);
 \draw[->] (vp) -- (u);
 \draw[->] (u) -- (v);
 \draw[->] (v) -- (up);
 \draw[->] (v) -- (x);
 \draw[->] (x) -- (u);
\end{scope}

\end{tikzpicture}
\caption{Schematic picture for the summands in the first three terms in~\eqref{eq:sconnect_N}. Left to right: $\tau_p(0,x)^2$, $U_p\supk{1}(0; u,v) \tau_p(u,x) \tau_p(v,x)$, and $U_p\supk{2}(0; u,v) \tau_p(u,x) \tau_p(v,x)$.
}
\label{fig:sconnect_expansion}
\end{figure}

\begin{proof}
The $N=0$ case holds by~\eqref{eq:sconnect_base} and the definitions of $R_p\supk{1}$ and $U_p\supk{1}$.
We prove the general case by induction on $N$.
Indeed, for any $N\ge 1$, 
we can upper bound the factor $\Pb_p(v \leadsto x \leadsto u)$ in the definition of $R_p\supN(x)$ using Lemma~\ref{lem:3pt_decomp} and the BK inequality. 
This gives
\begin{align}
R_p\supN(x)
&\le \sum_{ u\ne v } U_p\supk{N}(0; u,v) \times \biggl[ \tau_p(v,x) \tau_p(x,u) \nonumber \\
&\qquad\qquad +\sum_{ \substack{ u',v' } }
	 \tau_p(v ,u') \tau_p(u',v') \tau_p(v', u) \Pb_p(v' \leadsto x \leadsto u') \biggr] \nonumber \\
&= \sum_{ u\ne v }	U_p\supk{N}(0; u,v) \tau_p(u,x) \tau_p(v,x)
	+ R_p\supk{N+1}(x) ,
\end{align}
where 
in the second line we summed over all pairs of vertices $u',v'$ satisfying 
$u' \ne v'$, $u' \ne u$, and $v' \ne v$,
and
in the last line we interchanged the sums and used the definitions of $U_p\supk{N+1}$ and $R_p\supk{N+1}$.
Collecting terms then gives the desired result.
\end{proof}

We now take the limit $N\to \infty$.
The convergence requires the open triangle 
\begin{equation}
\triangle_p = \sup_{z \ne 0} \sum_{ x\ne y }
	\tau_p(0,x) \tau_p(x,y) \tau_p(y,z)
\end{equation}
to be strictly less than $1$.

\begin{proposition} \label{prop:sconnect_infty}
Let $G$ be a transitive graph.
Suppose $p > 0$ and $\triangle_p < 1$. 
Then for any $x \in V$, 
\begin{equation} \label{eq:sconnect_infty}
\Pb_p(0 \rightleftarrows x)
\le \tau_p(0,x)^2 
	+ \sum_{n=1}^\infty \sum_{ u\ne v }
		U_p\supk{n}(0; u,v) \tau_p(u,x) \tau_p(v,x).
\end{equation}
\end{proposition}

\begin{proof}
We want to take $N \rightarrow \infty$ in Lemma~\ref{lem:sconnect_N}.
Using monotone convergence for the middle term of~\eqref{eq:sconnect_N},
it suffices to prove $R_p\supk{N+1}(x) \rightarrow 0$ as $N \rightarrow \infty$.
We bound $\Pb_p(v \leadsto x \leadsto u) \le 1$ in the definition of $R_p\supk{N+1}$, so 
$R_p\supk{N+1}(x) 
\le \sum_{ u\ne v } U_p\supk{N+1}(0; u,v)$.
We will prove
\begin{equation} \label{eq:U_bound}
\sum_{ u\ne v } U_p\supk{N+1}(0; u,v)
\le \frac{ \triangle_p^{N+1} }{p} 
	\qquad (N \ge 0) ;
\end{equation}
this implies $R_p\supk{N+1}(x) \to 0$, since $p > 0$ and $\triangle_p < 1$.
For any $N$, by definition, Fubini's theorem, and transitivity, 
\begin{align}
\sum_{ u\ne v } U_p\supk{N+1}(0; u,v)
&= \sum_{ u\ne v }
	\sum_{ \substack{ u', v' } }
	U_p\supk{N}(0; u', v') \tau_p(v', u) \tau_p (u,v) \tau_p(v,u') \nonumber\\
&\le \sum_{ u'\ne v' } U_p\supk{N}(0; u', v') \triangle_p,
\end{align}
where in the first line we summed over all pairs of vertices $u',v'$ satisfying $u' \ne v'$, $u' \ne u$, and $v' \ne v$.
Combining this with the FKG estimate $\tau_p(v,0) \le \frac{1}{p} \tau_p(v,z)$, where $z$ is any neighbor of $0$, we then have
\begin{align}
\sum_{ u\ne v } U_p\supk{N+1}(0; u,v)
\le &\triangle_p^{N} \sum_{ u\ne v } U_p\supk{1}(0; u, v) \nonumber \\
\le &\frac{ \triangle_p^{N} }{p} \sum_{ u\ne v } 
	\tau_p(0,u) \tau_p (u,v) \tau_p(v,z)
\le \frac{ \triangle_p^{N+1} }{p} ,
\end{align}
as claimed. This concludes the proof.
\end{proof}

We can now prove Theorem~\ref{main_thm: susceptibility}.

\begin{proof}[Proof of Theorem~\ref{main_thm: susceptibility}]
It suffices to prove the first inequality of~\eqref{eq:susc}, as the second inequality follows by inserting~\eqref{eq:bubble_bdd} and $\triangle_p < 1$.
We use Proposition~\ref{prop:sconnect_infty}.
By summing~\eqref{eq:sconnect_infty} over all $x$ and invoking the definition of the bubble~\eqref{eq:def_bubble}, we have
\begin{equation}
\E_p |\cC^{\mathrm{sc}}_0| 
= \sum_{x\in V} \Pb_p(0 \rightleftarrows x)
\le \bubble_p +\sum_{n=1}^\infty\sum_{u\ne v} U_p\supk{n}(0;u,v) \sum_{x\in V} \tau_p(u,x)\tau_p(v,x).
\end{equation}
The last sum over $x$ is also bounded by $\bubble_p$ by Cauchy--Schwarz and transitivity.
Consequently, using~\eqref{eq:U_bound}, we have
\begin{equation}
 \E_p |\cC^{\mathrm{sc}}_0| 
 \le \bubble_p 
 \biggl(
 1+\sum_{n=1}^\infty\sum_{u\ne v} U_p\supk{n}(0;u,v)
 \biggr) \le \frac{\bubble_p}{p(1-\triangle_p)} ,
\end{equation}
as required.
\end{proof}

\section{Proof of Proposition~\ref{cor: spread-out_non-amenable}} 
\label{sec:proof-of-strict}

We recall the definition of $G_L$ from before. 
Given a graph $G = (V, E)$ and an integer $L \ge 1$, 
let $G_L$ be the graph with the same vertex set $V$ and edge set
\begin{equation}
 E_L := \{ \{x, y\} : x \neq y,\, 0 < d(x,y) \le L \}.
\end{equation}
We write $D_L := |B(x, L)| - 1$ for the degree of a vertex in $G_L$ (the ball has been defined in Subsection~\ref{subsec:structure}).
We start by proving the following estimate on $G_L$.

\begin{lemma} \label{lemma: uniform_spectral_gap}
Let $G$ be an infinite non-amenable transitive graph. 
Let $L \ge 1$ and let $P_L$ be the transition matrix of the simple random walk on $G_L$. 
There exists $\rho^* < 1$, depending only on the graph $G$, such that 
 \begin{equation} \label{ineq: uniform_spectral_gap}
 \sup_{L \ge 1} \|P_L\|_{2\rightarrow2} \le \rho^* < 1. 
 \end{equation}
\end{lemma}

\begin{proof}
 Let $F$ be a non-empty finite subset of $G$. 
 Define $f(x) := \sum_{z \in B(x, L-1)} \mathds{1}_{F}(z)$. Recall the definition of $\iota_E(G)$ from~\eqref{def: amenable}. Then, 
 \begin{align} \label{ineq: first_coarea}
 &\sum_{\{x, y\} \in E} \left|f(x) - f(y)\right| = \int_{0}^{\infty} |\partial_E\left(\{ f \ge t\}\right)| dt = \sum_{k = 1}^{\infty} | \partial_E\left( \{ f \ge k\} \right) | \nonumber \\
 &\quad\ge \iota_E(G) \sum_{k = 1}^{\infty} |\{ f \ge k \}| = \iota_E(G) \sum_{x \in V} f(x) = \iota_E(G) \, |F| \,|B(x, L-1)|,
 \end{align}
 where the first identity is usually referred to as \emph{the coarea formula}, which is justified by Fubini's theorem, the inequality step follows from the definition of $\iota_E(G)$, and the last identity follows from the fact that $\sum_{x}f(x) = |F| |B(x, L{-}1)|$.
On the other hand, by the coarea formula,
 \begin{align} \label{ineq: second_coarea}
 \sum_{\{x,y\} \in E_L} |\mathds{1}_{F}(x) - \mathds{1}_{F}(y)| = |\partial_{E_L}F|. 
 \end{align}
 Moreover, for every $\{x, y\} \in E$, since $|B(x, L-1)| = |B(y, L-1)|$ and $B(y, L-1) \subseteq B(x, L)$, we have 
 \begin{align}
 \left|f(x) - f(y)\right| = \,&\left|\sum_{z \in B(x, L-1)} \mathds{1}_{F}(z) - \sum_{z \in B(y, L-1)} \mathds{1}_{F}(z)\right| \nonumber \\
 = \,&\left|\sum_{z \in B(x, L)} \mathds{1}_{F}(z) \left( \mathds{1}_{B(x,L-1)}(z) - \mathds{1}_{B(y,L-1)}(z)\right)\right| \nonumber \\
 = \,&\left|\sum_{z \in B(x, L)} \left(\mathds{1}_{F}(z) - \mathds{1}_{F}(x)\right) \left( \mathds{1}_{B(x,L-1)}(z) - \mathds{1}_{B(y,L-1)}(z)\right)\right| \nonumber \\
 \le \,&\sum_{z \in B(x, L)} \left|\mathds{1}_{F}(z) - \mathds{1}_{F}(x)\right|.
 \end{align}
 Therefore, by summing over both orientations of every edge in $E$, 
 \begin{align} \label{ineq: compare_two_tvs}
 2 \sum_{\{x,y\} \in E}\left|f(x) - f(y)\right| \le 2 D \sum_{\{x,y\}\in E_L} \left|\mathds{1}_F(x) - \mathds{1}_F(y)\right|.
 \end{align}
 It follows from~\eqref{ineq: first_coarea},~\eqref{ineq: second_coarea}, and~\eqref{ineq: compare_two_tvs} that there exists $c > 0$ such that
 \begin{align}
 \frac{ \vert \partial_{E_L} F \vert }{D_L |F|} \ge \frac{\iota_E(G)\,|B(x, L{-}1)|}{D \cdot D_L} \ge c.
 \end{align}
Here $D$ is the degree of $G$, and we used $D_L\le D|B(x,L-1)|$.
 Thus, by Cheeger's inequality (see Theorem~6.7 and Proposition~6.6 of~\cite{LyonsPeres16}), we have
 $\|P_L\|_{2\rightarrow2} \le \sqrt{1-c^2} < 1 $ for every integer $L \ge 1$.
\end{proof}

We will also use the following general lemma.

\begin{lemma} \label{claim: (1+epsilon)/d_is_supercritical_spread_out}
 For a transitive graph $H = (V_H, E_H)$, let $d$ be its vertex degree. 
 Let $Q = Q(H)$ denote the transition matrix of the simple random walk on $H$. Let $\rho = \|Q\|_{2\rightarrow2}$.
 \begin{enumerate}
 \item[$(i)$] If $p$ satisfies $pd\rho < 1$ and
 \begin{equation} \label{ineq: sufficient_cond_super}
 pd - \frac{p^2d^{\frac{3}{2}}}{1-pd\rho} - \frac{p^2d}{2(1-pd\rho)^2} > 1,
 \end{equation}
 then $p \ge p_c(H)$.
 \item[$(ii)$] If $p$ satisfies $pd\rho < 1$, then
 \begin{equation} \label{ineq: triangle_upper_bound}
 \triangle_p \le \frac{p\sqrt{d}}{(1 - pd\rho)^3}. 
 \end{equation}
 \end{enumerate}
 \end{lemma}

\begin{proof}[Proof of Proposition~\ref{cor: spread-out_non-amenable} assuming Lemma~\ref{claim: (1+epsilon)/d_is_supercritical_spread_out}]
We will apply Theorem~\ref{main_thm: strict} to $G_L$.
By Lemma~\ref{lemma: uniform_spectral_gap}, we can choose $t, s$ such that $1 < t < s < (\rho^*)^{-1}$.
It suffices to prove,
 for all sufficiently large $L$, that
 \begin{align} \label{ineq: two_conditions_spread_out}
 p_c(G_L) \le \frac{t}{D_L} 
 \qquad \text{and} \qquad \triangle_{s / D_L} < 1.
 \end{align}
 For the first item, 
 we apply Lemma~\ref{claim: (1+epsilon)/d_is_supercritical_spread_out}(i) to $G_L$ with $p = t/D_L$. 
 In this case, condition~\eqref{ineq: sufficient_cond_super} translates to
 \begin{equation}
 t - \frac{t^2}{\sqrt{D_L}(1 - t\rho^*)} - \frac{t^2}{2D_L(1-t\rho^*)^{2}} > 1,
 \end{equation}
 which is true since $t \in (1, 1/ \rho^*)$ and $D_L \to \infty$ as $L\to \infty$.
For the second item, we apply Lemma~\ref{claim: (1+epsilon)/d_is_supercritical_spread_out}(ii) to $G_L$ with $p = s/D_L$ to get
 \begin{equation}
 \triangle_{s/D_L} \le \frac{s}{\sqrt{D_L}(1 - s\rho^*)^3}.
 \end{equation}
 This is strictly less than $1$ since $s \in (1, 1/ \rho^*)$ and $D_L \to \infty$.
We thus have 
 verified both inequalities in~\eqref{ineq: two_conditions_spread_out} for all sufficiently large $L$. This concludes the proof.
\end{proof}

It remains to prove Lemma~\ref{claim: (1+epsilon)/d_is_supercritical_spread_out}.

 \begin{proof}[Proof of Lemma~\ref{claim: (1+epsilon)/d_is_supercritical_spread_out}]
 Fix $\epsilon \in (0, 1)$.
 We mark each vertex independently with probability $\epsilon$. The marking is independent of the percolation configuration. Fix an arbitrary vertex $0 \in V$. 
 We write ``$x \leftrightarrow \mathsf{mark}$'' if $x$ is connected to a marked vertex and write ``$x \leftrightarrow y$ off $z$'' if there is an open path connecting $x$ and $y$ without using edges incident to $z$. 
 For $v \in V_H$, we define the events
 \begin{equation}
 M_v := \left\{ v \leftrightarrow \mathsf{mark} \right\}, \quad \text{and} \quad
 B_v := \left\{ v \leftrightarrow \mathsf{mark} \text{ off } 0 \right\}.
 \end{equation}
 For $v \sim 0$, define
 \begin{equation}
 A_v := \left\{ \{0, v\} \text{ is open} \right\} \cap B_v. 
 \end{equation}
 Since the two events in the definition of $A_v$ are independent, we have
 \begin{equation} \label{eq:A_B}
 \Pb_p(A_v) = p \Pb_p(B_v) \quad \text{and} \quad \Pb_p(A_v \cap A_u) = p^2 \Pb_p(B_v \cap B_u)
 \end{equation}
for distinct neighbors $u,v$ of $0$.
By conditioning on $\cC_0^{\mathrm{un}}$,
 \begin{equation} \label{limit: epsilon_theta_p}
 \Pb_p(M_0) = \E \left[ 1 - (1 - \epsilon)^{|\cC_0^{\mathrm{un}}|} \right] \rightarrow \theta(p) \ \text{ as $\epsilon \rightarrow 0$,}
 \end{equation}
 where the limit follows from the dominated convergence theorem. 
To prove part (i) of the lemma, we will show that $\Pb_p(M_0)$ is uniformly positive as $\eps \to 0$ under our hypotheses on $p$. 
This will imply $\theta(p) > 0$ and hence $ p \ge p_c$ at that $p$.
We start from the observation that
 \begin{equation} \label{ineq: approxi_theta}
 \Pb_p(M_0) \ge \sum_{v \sim 0}\Pb_p(A_v) - \frac{1}{2} \sum_{\substack{u, v \sim 0\\u \neq v}} \Pb_p(A_v\cap A_u) ,
 \end{equation}
 where the $1/2$ comes from counting both $(u,v)$ and $(v, u)$. 

 We first lower bound $\Pb_p(B_v)$ and upper bound $\Pb_p(B_v \cap B_u)$. 
 To lower bound $\Pb_p(B_v)$, observe that on the event $M_v \setminus B_v$, 
 there must exist two paths using disjoint edges, one from $v$ to $0$ and one from $0$ to a marked vertex. 
 By the BK inequality,
 \begin{equation} \label{ineq: B}
 \Pb_p(M_v) - \Pb_p(B_v) \le \tau_p(v, 0) \Pb_p(M_0).
 \end{equation}
 The event $B_v \cap B_u$ can be split into two events: (1) $B_v$ and $B_u$ both occur using disjoint edges and marks, and (2) there is a vertex $z$ such that $v \leftrightarrow z$ off $0$, $u \leftrightarrow z$ off $0$, and $z \leftrightarrow \mathsf{mark}$ off $0$, and all three events use disjoint edges and marks. Therefore, by the BK inequality, we have 
 \begin{equation} \label{ineq: B_coralation}
 \Pb_p(B_v \cap B_u) \le \Pb_p(B_v) \Pb_p(B_u) + \sum_{z \in V_H} \Pb_p(B_z) \tau_p(u, z) \tau_p(v, z).
 \end{equation}

 By transitivity, we have $\Pb_p(M_v) = \Pb_p(M_0)$ for all $v \in V_H$.
 Combining~\eqref{eq:A_B},~\eqref{ineq: approxi_theta},~\eqref{ineq: B},~\eqref{ineq: B_coralation}, and $\Pb_p(B_v) \le \Pb_p(M_v)$, we have that 
 \begin{equation}
 \Pb_p(M_0) 
 \ge p\Pb_p(M_0)\sum_{v \sim 0} (1 - \tau_p(v, 0)) - \frac{p^2}{2} \sum_{\substack{u, v \sim 0\\u \neq v}} \biggl( \Pb_p(M_0)^2 + \Pb_p(M_0) \sum_{z\in V_H} \tau_p(u, z)\tau_p(v, z) \biggr).
 \end{equation}
 Using $\epsilon > 0$, we can divide both sides by 
 $\Pb_p(M_0) > 0$. 
 Then rearranging gives
 \begin{equation} \label{ineq: raw_cond_super}
 \frac{p^2}{2} \Pb_p(M_0) d^2 \ge pd - p\sum_{v \sim 0} \tau_p(v, 0) - \frac{p^2}{2} \sum_{\substack{u, v \sim 0 \\ u \neq v}} \sum_{z \in V_H} \tau_p(u, z) \tau_p(v, z) - 1.
 \end{equation}
 
 We next use $pd\rho < 1$.
 Let $R := \sum_{n \ge 0} (pd)^n Q^n$.
 Then $\|R\|_{2\rightarrow2} = 1/(1-pd\rho)$ and $\|R-I\|_{2\rightarrow2} = pd\rho/(1-pd\rho)$.
 Using $\|\sum_{z \sim 0}\delta_z\|_2 = \sqrt{d}$
 and using the fact that $\tau_p(u, v) \le (R-I) (u, v)$ for every $u\neq v$, we have 
 \begin{equation}
 \sum_{v \sim 0}\tau_p(0, v) \le \left\langle \delta_0, \left(R-I\right) \sum_{v \sim 0}\delta_v \right\rangle
 \le \left\|R-I\right\|_{2\rightarrow2} \; \left\|\sum_{v \sim 0} \delta_v\right\|_{2} \le \frac{pd^{3/2}\rho}{1-pd\rho}
 \le \frac{pd^{3/2}}{1-pd\rho} .
 \end{equation}
 Since $\tau_p(u, v) \le R (u, v)$ for every $u,v$, 
 \begin{equation}
 \sum_{\substack{u, v \sim 0 \\ u \neq v}} \sum_{z \in V_H} \tau_p(u, z) \tau_p(v, z) \le \sum_{z \in V_H} \left(\sum_{u \sim 0} \tau_p(u, z)\right)^2 \le \left\|R \left( \sum_{u \sim 0} \delta_u \right)\right\|_2^2 
 \le \frac{d}{(1 - pd\rho)^2}.
 \end{equation}
 Inserting these into~\eqref{ineq: raw_cond_super}, we obtain
 \begin{equation} 
 \frac{p^2d^2 }{2} \Pb_p(M_0) 
\ge pd - \frac{p^2 d^{3/2}}{1-pd\rho}
 - \frac{p^2d}{2(1 - pd\rho)^2}
 - 1,
 \end{equation}
which is strictly positive by hypothesis~\eqref{ineq: sufficient_cond_super}.
Letting $\epsilon \rightarrow 0$ and using~\eqref{limit: epsilon_theta_p},
we get $\theta(p) > 0$, so $p \ge p_c$, proving the first part of the lemma.

 Finally, we prove part (ii) of the lemma.
 Note that $\|Q\delta_z\|_2 = d^{-1/2}$.
 For any $z \in V_H$, by the Cauchy--Schwarz inequality, 
 \begin{align}
 \sum_{x\neq y} \tau_p(0, x)\tau_p(x, y)\tau_p(y, z) 
 &\le \left\langle \delta_0, R (R-I) R \delta_z\right\rangle \le \|R\|_{2\rightarrow2}^2 \|(R-I)\delta_z\|_2 \nonumber \\
 &= \|R\|_{2\rightarrow2}^2 \|pdRQ\delta_z\|_2 \le \frac{p\sqrt{d}}{(1 - pd\rho)^3}. 
 \end{align}
 We then take the supremum over $z$ to get the desired bound on $\triangle_p$.
 \end{proof}

\appendix

\section{Bubble diagram on $\Z^2$}
\label{app:bubble}

Recall $\tau_p(x,y) = \Pb_p(x \leftrightarrow y)$ denotes the two-point function for (undirected) Bernoulli percolation.
In this appendix, we give a short proof of the following classical result.

\begin{lemma}
For $\Z^2$,
we have $\bubble_{p_c} = \sum_{x\in \Z^2} \tau_{p_c}(0,x)^2 = \infty$.
\end{lemma}

\begin{proof}
Let $\theta(p) = \Pb_p(0\leftrightarrow\infty)$.
For every $p > p_c$, decomposing according to the first exit from $B_n = [-n,n]^2$ and applying the BK inequality gives
 \begin{equation}
 \theta(p) \le \sum_{x \in \partial B_n} \Pb_p(0 \xlongleftrightarrow[]{B_n} x) \theta(p), 
 \end{equation}
 where $0 \xleftrightarrow[]{B_n} x$ denotes the event that $0$ is connected to $x$ by a path contained in $B_n$.
 Note that this event depends only on finitely many edges, and hence its probability is continuous in $p$.
 Letting $p \downarrow p_c$ then gives $\sum_{x \in \partial B_n} \tau_{p_c}(0,x) \ge 1$.
 Now, by the Cauchy--Schwarz inequality, 
 \begin{equation}
 \bubble_{p_c} = 
 \sum_{n = 0}^{\infty} \sum_{x \in \partial B_n} \tau_{p_c}(0,x)^2 
 \ge \sum_{n = 0}^{\infty} \frac{(\sum_{x \in \partial B_n} \tau_{p_c}(0,x) \cdot 1)^2}{|\partial B_n|} 
 \ge \sum_{n=1}^\infty \frac{c}{n} = \infty,
 \end{equation}
 as desired.
\end{proof}

\section{Continuity of the open triangle diagram}
\label{app:cont}

Recall that
\begin{align} \label{eq:p_triangle}
p_{\triangle} 
&= \sup \Bigl\{ p : 
 \triangle_p = \sup_{z \ne 0} \sum_{ x\ne y }
	\tau_p(0,x) \tau_p(x,y) \tau_p(y,z)
 < \infty \Bigr\}
 \nonumber \\
&= \sup \Bigl\{ p : 
 \sup_{z \in V} \sum_{ x, y\in V }
	\tau_p(0,x) \tau_p(x,y) \tau_p(y,z)
 < \infty \Bigr\} 
\end{align}
denotes the finite-triangle threshold.

\begin{lemma}
The map $p \mapsto \triangle_p$ is continuous on $(0,p_\triangle)$.
\end{lemma}

\begin{proof}
Let $\triangle_p(z) = \sum_{x \neq y} \tau_p(0, x) \tau_p(x, y) \tau_p(y, z)$, so that $\triangle_p = \sup_{z\ne 0} \triangle_p(z)$.
It suffices to prove that $p \mapsto \triangle_p$ is continuous on $[\eps,p_\triangle - \eps]$ for all $\eps > 0$.
For simplicity, we write $q = p_\triangle - \eps$.
By the second equality in~\eqref{eq:p_triangle}, we know the triangle condition holds at $q$,
so the main result of~\cite{10.1214/09-AIHP352} implies that
\begin{equation}
\triangle_p = \max_{ 0 < |z| \le R_\eps} \triangle_p(z)
\end{equation}
for all $p \in [\eps, q]$ for some $R_\eps > 0$. 
(This can be established by taking $R_\eps$ large enough so that $\sup_{|z| > R_\eps} \triangle_q(z) 
< \triangle_{\eps}$.)
The desired result will follow by applying dominated convergence, using the fact that $\triangle_p(z)$ is continuous in $p$ for each fixed $z$.

% \textcolor{blue}{
To prove said continuity,
let $\tau_p^{(n)}(a,b)$ denote the probability of connectivity within a finite ball $B(0,n)$.
By the BK inequality, for $p \le q$ we know
$
0\le \tau_p(a,b)-\tau_p^{(n)}(a,b)
\le \sum_{x\notin B(0,n)}\tau_q(a,x)\tau_q(x,b) .
$
Since the right-hand side is bounded by $\sum_{x\in V}\tau_q(a,x)^2 < \infty$ by the Cauchy--Schwarz inequality, 
we learn that $\tau_p(a,b)$ is the uniform limit of the continuous functions $\tau_p\supk{n}(a,b)$, as $n \to \infty$.
Therefore, $\tau_p(a,b)$ is continuous in $p \le q$, 
and, from $\triangle_q(z)<\infty$ and dominated convergence, 
we know that $\triangle_p(z)$ is continuous in $p \le q$ as well.
This concludes the proof.
% }
%
%is continuous in $p$ as a uniform limit of continuous functions.
%Since $\triangle_q(z)<\infty$, dominated convergence then gives continuity of $\triangle_p(z)$ in $p$ for each fixed $z$.
%\marginnote{Dominated convergence requires continuity of $\tau_p(a,b)$, I think this sentence should be expanded. --Leo}
%Then, since $\triangle_p(z)$ is continuous in $p$ for each fixed $z$ by dominated convergence, the desired result follows.
\end{proof}

\section{Extensions on the main results}
\label{app:extensions}

In this appendix, we compare our model with a more general model studied by Grimmett~\cite{MR1824275} and explain which of our results extend.
Let $\widehat{\Pb}$ denote the law of the general model. 
The two arrows on each edge may depend on each other and have different probabilities of being open, while arrow pairs on distinct undirected edges are independent. For each edge ${x,y}$, write
\begin{equation}
	p_{xy}=\widehat{\Pb}(\omega_{xy}=1),
	\qquad
	q_{\{x,y\}}=\widehat{\Pb}(\omega_{xy}=\omega_{yx}=1).
\end{equation}
The probabilities $p_{xy}$ and $p_{yx}$ may differ, and the covariance of the two arrows is $q_{\{x,y\}}-p_{xy}p_{yx}$. 
Our model corresponds to $p_{xy}=p$ for every directed edge and $q_{\{x,y\}}=p^2$ for every undirected edge.

% \commentYL{
% I don't think we say much when $p_{xy} \ne p_{yx}$ is allowed? How about dropping the asymmetry?

% {\color{purple}YS: It seems Theorem~C.3 holds for possibly unequal marginals. Maybe it is fine to keep it like this for now?}
% }

Theorem~\ref{main_thm: equality} extends to models with a common marginal and nonnegative covariance: 
\begin{theorem}
Let $G$ be an infinite, connected, locally finite, amenable transitive graph. Fix $p\in[0,1]$, and let $\widehat{\Pb}$ be a law under which $p_{xy}=p$ for every directed edge and $q_{\{x,y\}}\ge p^2$ for every undirected edge. Then, for every $x\in V$,
\begin{equation}
	\widehat{\Pb}(| \cC_x^{\mathrm{sc}}|=\infty)>0 \quad 
	\Longleftrightarrow \quad \Pb_p(|\cC_x^{\mathrm{un}}|=\infty)>0,
\end{equation}
where $\Pb_p$ denotes Bernoulli bond percolation with parameter $p$.
\end{theorem}

\begin{proof}
Lemma~\ref{lemma:directed} remains valid, and the condition $q_{{x,y}}\ge p^2$ gives the FKG inequality. 
% \textcolor{blue}{
Indeed, one can check the FKG lattice condition holds for the two arrows on each undirected edge and then use independence between different edges to extend the FKG inequality to the entire graph.
% }
Whenever Bernoulli bond percolation at $p$ percolates, its infinite cluster is almost surely unique by~\cite{BurtonKeane89}, so the argument proving Theorem~\ref{main_thm: equality} gives
$\widehat{\Pb}(|\cC_x^{\mathrm{sc}}|=\infty)>0$.
The reverse implication follows from
$\cC_x^{\mathrm{sc}}\subseteq\cC_x^{\mathrm{out}}$
and Lemma~\ref{lemma:directed}.
This argument does not cover unequal opposite marginals or negative correlation.
\end{proof}

Theorem~\ref{main_thm: structure} and Corollary~\ref{cor: classification_of_vertices} extend to models with a common marginal and arbitrary correlation between opposite arrows:
\begin{theorem}
	Let $G$ be an infinite, connected, locally finite, transitive graph. Fix $p\in[0,1]$. Assume that $p_{xy}=p$ for every directed edge. 
	Assume further that $\widehat{\Pb}$ is invariant under a transitive group of graph automorphisms and that Bernoulli bond percolation at $p$ almost surely has a unique infinite cluster. 
	Then the conclusions of Theorem~\ref{main_thm: structure} and Corollary~\ref{cor: classification_of_vertices} hold, 
	with the directed clusters sampled under $\widehat{\Pb}$ and the undirected clusters under $\Pb_p$.
\end{theorem}

Note that the invariance under a transitive group of graph automorphisms does not force $q_{\{x, y\}}$ to be the same for all edges. The reason is that although there is only one vertex-orbit, there can be multiple edge-orbits, and $q$ can vary between edges in different orbits. Consider the ladder graph $\Z \times K_2$, the horizontal and vertical edges belong to different edge-orbits, so $q$ can differ between them.

\begin{proof}
Propositions~\ref{prop-1} and~\ref{prop-2} apply unchanged: their explorations inspect at most one arrow per undirected edge, so each newly revealed arrow is Bernoulli$(p)$ and independent of the previously revealed arrows. 
The coupling with Bernoulli bond percolation therefore remains valid.
In the rest of the proof of Theorem~\ref{main_thm: structure} and Corollary~\ref{cor: classification_of_vertices}, FKG is used only in \eqref{ineq: FKG_structure}, to show that $\cI$ is nonempty. 
This step can be replaced by the following.

By ergodicity, it suffices to prove that $\widehat{\Pb}(\cI\ne\varnothing)>0$.
Suppose instead that $\cI=\varnothing$ almost surely.
Since Bernoulli bond percolation at $p$ has an infinite cluster, 
Lemma~\ref{lemma:directed} and ergodicity imply that $\cI_i$ and $\cI_o$ are both nonempty almost surely.
Choose a shortest undirected path $\gamma$ from $\cI_i$ to $\cI_o$. 
Its interior vertices almost surely belong to $\cI_\varnothing$.
There is an infinite inward ray into its initial vertex, entirely contained in $\cI_i$, and an infinite outward ray from its final vertex, entirely contained in $\cI_o$. 
Hence both rays avoid every edge of $\gamma$.
By countability, there is a deterministic finite path $\pi$ that has positive probability of having infinite inward and outward rays at its endpoints, using no edge of $\pi$.
This event is independent of the event that the arrows along $\pi$, directed from its initial to its final vertex, are all open, which has probability $p^{|\pi|}>0$.
% Independently requiring the arrows along $\pi$, directed from its initial to its final vertex, to be open has probability $p^{|\pi|}>0$.
On their intersection, the initial vertex has both an infinite inward cluster and an infinite outward cluster, so it belongs to $\cI$, a contradiction.
\end{proof}
% The remaining proof applies unchanged. 

The bound in Theorem~\ref{main_thm: susceptibility} extends to the general model by comparison with independent arrows. Set
\begin{equation}
	r=\sup_{{u,v}\in E}\max\{p_{uv},p_{vu},\sqrt{q_{\{u,v\}}}\}.
\end{equation}
As in Theorem~\ref{main_thm: susceptibility}, $\bubble_r$ and $\triangle_r$ are the diagrams for Bernoulli bond percolation at parameter $r$.

\begin{theorem}
Let $G=(V,E)$ be an infinite, connected, locally finite, transitive graph. Then, for every $x,y\in V$,
\begin{align}\label{eq: remark_G}
\widehat{\Pb}(x\rightleftarrows y)
\le \Pb_r(x\rightleftarrows y).
\end{align}
If $r>0$ and $\triangle_r<1$, then for every $x\in V$,
\begin{equation}
	\widehat{\E}|\cC_x^{\mathrm{sc}}|
	\le \frac{\bubble_r}{r(1-\triangle_r)}
	\le \frac{2}{r^2(1-\triangle_r)}.
\end{equation}
In particular, if the arrows have a common marginal $p$ and nonpositive covariance between opposite arrows, then $q_{\{u,v\}}\le p^2$ for every edge and hence $r=p$. Thus, the upper bound on $\widehat{\E}|\cC_x^{\mathrm{sc}}|$ holds at parameter $p$ whenever $p>0$ and $\triangle_p<1$.
\end{theorem}

% The bound in Theorem~\ref{main_thm: susceptibility} extends to models with a common marginal $p$ and nonpositive correlation between opposite arrows. More generally, set
% $r=\sup_{{x,y}\in E}\max\{p_{xy},p_{yx},\sqrt{q_{\{x,y\}}}\}.$
% Then,
% \begin{align} \label{eq: remark_G}
% \widehat{\Pb}(x\leftrightarrows y)
% \le \Pb_r(x\leftrightarrows y), \qquad \forall x, y \in V,
% \end{align}
% where $\widehat{\Pb}$ denotes the law of the general model and $\Pb_r$ denotes our model, in which each arrow is independently open with probability $r$. Consequently, the bound in Theorem~\ref{main_thm: susceptibility} holds with $p$ replaced by $r$ whenever $\triangle_r<1$.

\begin{proof}
To prove \eqref{eq: remark_G}, fix the status of all arrows except the pair associated with $\{u,v\}$. 
As a function of $X=\omega_{uv}$ and $Y=\omega_{vu}$, 
the indicator of $\{x\leftrightarrows y\}$ can only be $0$, $1$, $X$, $Y$, or $XY$. 
The case $X \vee Y$ is impossible: 
if the indicator were $X \vee Y$, adding either arrow would merge two strong connected clusters, that is, this means that before looking at the oriented edges between $u$ and $v$, we already have both $\{u\leadsto v\}$ and $\{v\leadsto u\}$, hence occurrence of $\{x\leftrightarrows y\}$ is not affected by the state of these edges.
The conditional probability of $\{x\leftrightarrows y\}$ is therefore $0$, $1$, $p_{uv}$, $p_{vu}$, or $q_{\{u,v\}}$. 
Since $p_{uv},p_{vu}\le r$ and $q_{\{u,v\}}\le r^2$, 
replacing this pair by independent arrows of marginal $r$ cannot decrease this probability.
Independence between distinct edge-pairs allows successive replacements, which gives \eqref{eq: remark_G} (more precisely, write the connection probabilities as limits over finite subgraphs, and do finitely many successive replacements within each one).
\end{proof}

\section*{Use of AI}

The two examples to which Theorem~\ref{main_thm: strict} applies were first suggested by Claude.
Other AI tools have been used for grammar checking and proofreading.

\section*{Acknowledgments}

We thank Omer Angel, Daniel de la Riva, Geoffrey R. Grimmett, Tom Hutchcroft, Christoforos Panagiotis, and Gábor Pete for helpful comments and discussions.
The work of YL was supported in part by the National Natural Science Foundation of China (Grant Nos.~12595294 and 12231002)
and the New Cornerstone Science Foundation (Grant No.~NCI202501). 
LR has been supported by FAPESP grants 2023/13453-5 and 2025/27064-6 and CNPq grant 408590/2024-6.
The work of YS was supported in part by NSERC.

\bibliographystyle{plain}
\bibliography{ref.bib}

\end{document}